\documentclass{amsart}[12pt]
\newtheorem{theorem}{Theorem}[section]
\newtheorem{lemma}[theorem]{Lemma}
\newtheorem{proposition}[theorem]{Proposition}
\newtheorem{corollary}[theorem]{Corollary}

\theoremstyle{definition}
\newtheorem{definition}[theorem]{Definition}

\newcommand{\End}{\mathop{\mathrm{End}}\nolimits}
\newcommand{\Ad}{\mathop{\mathrm{Ad}}\nolimits}
\newcommand{\ad}{\mathop{\mathrm{ad}}\nolimits}

\newcommand{\spin}{\mathop{\mathrm{spin}}\nolimits}
\newcommand{\Kil}{\mathop{\mathrm{Kil}}\nolimits}

\renewcommand{\div}{\mathop{\mathrm{div}}\nolimits}
\newcommand{\grad}{\mathop{\mathrm{grad}}\nolimits}

\newcommand{\cL}{{\mathcal L}}
\newcommand{\cF}{{\mathcal F}}
\newcommand{\cS}{{\mathcal S}}
\newcommand{\cT}{{\mathcal T}}

\newcommand{\bC}{{\mathbb C}}
\newcommand{\bR}{{\mathbb R}}
\newcommand{\fg}{{\mathfrak g}}
\newcommand{\fk}{{\mathfrak k}}
\newcommand{\ft}{{\mathfrak t}}
\newcommand{\fm}{{\mathfrak m}}

\newcommand{\<}{\langle}
\renewcommand{\>}{\rangle}
\newcommand{\dia}{\diamond}
\renewcommand{\a}{\alpha}
\renewcommand{\b}{\beta}

\renewcommand{\o}{\omega}
\newcommand{\g}{\gamma}
\newcommand{\G}{\Gamma}
\newcommand{\D}{\Delta}
\renewcommand{\k}{\kappa}
\renewcommand{\l}{\lambda}
\renewcommand{\d}{\delta}

\newcommand{\var}{\varphi}

\newcommand{\jd}{J_\diamond}
\newcommand{\zd}{Z_\diamond}

\numberwithin{equation}{section}

\begin{document}

\title[Dirac operators for coadjoint orbits]{Dirac operators for coadjoint orbits of compact Lie groups}

\author{Marc A. Rieffel}
\address{Department of Mathematics\\
University of California\\
Berkeley, CA\ \ 94720-3840}
\curraddr{}
\email{rieffel@math.berkeley.edu}
\thanks{The research reported here was supported in part by 
National Science Foundation Grants DMS-0500501 and DMS-0753228.}
\dedicatory{In celebration of the sixtieth birthday of Joachim Cuntz}


\subjclass[2000]{Primary 53C30; Secondary 46L87, 58J60, 
53C05}
\keywords{Dirac operators, coadjoint orbit, connection, Levi--Civita, almost Hermitian, Clifford algebra, spinor}

\large{

\begin{abstract}
The coadjoint orbits of compact Lie groups carry many K\"ahler structures, which include a Riemannian metric and a complex structure.  We provide a fairly explicit formula for the Levi--Civita connection of the Riemannian metric, and we use the complex structure to give a fairly explicit construction of  
a canonical Dirac operator for the Riemannian metric, in a way that avoids use of the $\spin^c$ groups.  Substantial parts of our results apply to compact almost-Hermitian homogeneous spaces, and to other connections besides the Levi--Civita connection.  For these other connections we give a criterion that is both necessary and sufficient for their Dirac operator to be formally self-adjoint.

We hope to use the detailed results given here to clarify statements in the literature of high-eneregy physics concerning ``Dirac operators'' for matrix algebras that converge to coadjoint orbits.  To facilitate this we employ here only global methods --- we never use local coordinate charts, and we use the cross-section modules of vector bundles.
\end{abstract}

\maketitle

\section*{Introduction}

In the literature of theoretical high-energy physics one finds statements along the lines of ``matrix algebras converge to the sphere'' and ``here are the Dirac operators on the matrix algebras that correspond to the Dirac operator on the sphere''.  But one also finds that at least three inequivalent types of Dirac operator are being used in this context.  See, for example, \cite{AIM,Aok,BIm,BKV,CWW,GP2,HQT,Yd1,Ydr} and the references they contain, as well as \cite{Ngo} which contains some useful comparisons.  In \cite{R6,R7,R21} I provided definitions and theorems that give a precise meaning to the convergence of matrix algebras to spheres.  These results were developed in the general context of coadjoint orbits of compact Lie groups, which is the appropriate context for this topic, as is clear from the physics literature.  I seek now to give a precise meaning to the statement about Dirac operators.  For this purpose it is important to have a detailed understanding of Dirac operators on coadjoint orbits, in a form that is congenial to the non-commutative geometry that is used in treating the matrix algebras.  This means, for example, that one should work with the modules of continuous sections of vector bundles, rather than the points of the bundles themselves, and one should not use local coordinate charts. (Standard module frames are very useful to us in this connection.)  The purpose of this paper is to give such a congenial detailed understanding of Dirac operators on coadjoint orbits.

Let $G$ be a connected compact semisimple Lie group, with Lie algebra $\fg$.  Let $\fg'$ denote the vector-space dual of $\fg$, and let $\mu \in \fg'$ with $\mu \neq 0$.  The coadjoint orbit of $\mu$ can be identified with $G/K$ where $K$ is the stability subgroup of $\mu$.  Then $\mu$ determines a $G$-invariant K\"ahler structure on $G/K$, which includes a Riemannian metric and a complex structure \cite{BFR}. This complex structure determines a canonical $\spin^c$ structure on $G/K$. A principal objective of this paper is to give a reasonably explicit construction of the Dirac operator for this $\spin^c$ structure. Toward this objective we obtain in Section~\ref{sec3} a reasonably specific formula for the Levi--Civita connection for the Riemannian metric determined by $\mu$.  (The only place I have seen this Levi--Civita connection discussed in the literature is in section 7 of \cite{BIL}, where the context is not sufficiently congenial to non-commutative geometry for my purposes.)  Our construction of the Dirac operator, along the lines given in \cite{Ply, GVF, Vrl}, never involves the 
$\spin^c$ groups, with their attendant complications.  We will also consider Dirac operators for $\spin^c$ structures obtained by twisting the canonical one.

We remark that coadjoint orbits are always $\spin^c$ manifolds, but many are not spin manifolds.  See \cite{BIL, Owc, DlN, LwM} for interesting specific examples. But I have not found a description in the literature of exactly which coadjoint orbits are spin. (Though see remark 3.6 of \cite{Goe}.) We will not discuss here the charge conjugation that can be constructed for the Dirac operator coming from a spin structure, but underlying the spin structure on a coadjoint orbit that is spin will be one of the twisted $\spin^c$ structures that we consider, for the reasons indicated by definition 9.8 of \cite{GVF}.

But there are other $G$-invariant metrics of interest on $G/K$, the most obvious one coming from using the Killing form of $\fg$.  This metric will come from the K\"ahler structure on a coadjoint orbit only in the special case that $G/K$ is a symmetric space.  More generally, as is explained well on page~21 of \cite{Dui}, if the Levi--Civita connection for a Riemannian manifold commutes with a complex structure, then the Riemannian metric is part of a K\"ahler structure on the manifold.  But as explained in \cite{BFR}, if $G/K$ has a K\"ahler structure then $G/K$ must correspond to a coadjoint orbit.  The consequence of this is that if we want to treat Riemannian metrics such as that from the Killing form, and if we want to use a complex (or almost-complex) structure to construct the Dirac operator, then we must use connections that are not torsion-free.  But then we must be concerned with whether the corresponding Dirac operator is formally self-adjoint, as is usually desired.

To deal with this more general situation, we develop a substantial part of our results for the more general case in which $G/K$ is almost-Hermitian.  There are many more coset spaces $G/K$ that admit a $G$-invariant almost-Hermitian structure, beyond those that arise from coadjoint orbits.  In Theorem \ref{th7.1} we give a convenient criterion, in terms of the torsion, that is both necessary and sufficient, for the Dirac operator constructed using a connection compatible with a $G$-invariant almost-Hermitian structure, to be formally self-adjoint. Our criterion is very similar to the one given in the main theorem of \cite{Ikd}, which treats the case of homogeneous spaces that are spin. (See also \cite{FrS}.) The criterion in \cite{Ikd} is restated as proposition 3.1 of \cite{Agr}, which again treats homogeneous spaces that are spin, and focuses on ``naturally reductive'' Riemannian metrics. As we will indicate after Theorem 3.3, the metric from the K\"ahler structure of a coadjoint orbit is ``naturally reductive" exactly in the special case when the coadjoint orbit is a symmetric space. Also, our global techniques are different from the techniques of these two papers.

 Among the corollaries of our criterion we prove that for the canonical connection on an almost-Hermitian $G/K$ its Dirac operator is always formally self-adjoint.  In particular, this applies to coadjoint orbits when they are equipped with the Riemannian metric coming from the Killing form. (In this case the canonical connection often has non-zero torsion.)

It would be very interesting to know how the results in the present paper relate to those in \cite{Krh}.  In \cite{Krh} only one ``metric'' on a quantum flag manifold appears to be used, and my guess is that it corresponds to the Killing-form metric, and that the self-adjointness of the Dirac operator relates to our Corollary \ref{cor7.5}.  But I have not studied this matter carefully.  It would also be interesting to study the extent to which the results of the present paper can be extended to the setting of \cite{DlN}, or used in the setting immediately after equation 6.31 of \cite{LPS}.

The present paper builds extensively on the paper \cite{R22}, in which I gave a treatment of equivariant vector bundles, connections, and the Hodge--Dirac operator, for general $G/K$ with $G$ compact, in a form congenial to the framework of non-commutative geometry.  (The most recent arXiv version of \cite{R22} has important corrections and improvements compared to the published version.)

In Section~\ref{sec1} of the present paper we describe at the level of the Lie algebra the K\"ahler structure for a coadjoint orbit.  In Section~\ref{sec2} we obtain a general formula for the Levi--Civita connection for a $G$-invariant Riemannian metric on a coset space $G/K$ for $G$ compact.  In Section~\ref{sec3} we use results from Section~\ref{sec1} together with the general formula of Section~\ref{sec2} to obtain a rather specific formula for the Levi--Civita connection for the Riemannian metric of the K\"ahler structure on a coadjoint orbit.  At no point do we need to use the full structure theory of semisimple Lie algebras --- we only need the non-degeneracy of the Killing form.  In Section~\ref{sec5} we develop, at the level of the Lie algebra, the Clifford algebra and its spinor representation corresponding to the complex structure of an almost-Hermitian coset space; and then in Section~\ref{sec6} we use this to define the field of Clifford algebras, the spinor bundle, and the Dirac operator for an almost-Hermitian coset space $G/K$.  We also obtain there some of the basic properties of the Dirac operator.  Finally, in Section~\ref{sec7} we obtain the criterion mentioned above for when the Dirac operator will be formally self-adjoint, and we apply this criterion to the case of the Riemannian metric from the K\"ahler structure of a coadjoint orbit, and also to the case of the Riemannian metric from the Killing form.

A part of the research for this paper was carried out during a six-week visit I made to Scuola Internazionale Superiore di Studi Avanzati (SISSA) in Trieste, where Dirac vibrations are strong.  I am very appreciative of the stimulating atmosphere there, and of the warm hospitality of Gianni Landi and Ludwik Dabrowski during my visit.

I am very grateful to the referee for detailed comments on the first version of this paper, which in particular led to some important improvements.


\section{The canonical K\"ahler structure}
\label{sec1}

Let $G$ be a connected compact Lie group.  Let $\fg$ be its Lie algebra, and let $\Ad$ be the adjoint action of $G$ on $\fg$.  Let $\fg'$ be the vector-space dual of $\fg$, and let $\Ad'$ be the coadjoint action of $G$ on $\fg'$, that is, the dual of the action $\Ad$.  The coadjoint orbits are the orbits in $\fg'$ for the action $\Ad'$.  Let $\mu_{\dia} \in \fg'$, with $\mu_{\dia} \ne 0$.  We will obtain in this section quite explicit formulas for the restriction to the tangent space at $\mu_{\dia}$ of the canonical K\"ahler structure on the coadjoint orbit through $\mu_{\dia}$.  We will usually mark with a $\dia$ the various pieces of structure that depend canonically on the choice of $\mu_{\dia}$.  In Sections~\ref{sec2} and \ref{sec3} we will see how to construct the K\"ahler structure on the whole coadjoint orbit through $\mu_{\dia}$.  This K\"ahler structure includes a Riemannian metric and a complex structure.  In Section~\ref{sec6} we will construct the Dirac operator for this Riemannian metric on the canonical $\spin^c$ structure determined by the complex structure.

Since the center of $G$ leaves all the points of $\fg'$ fixed, we do not lose generality by assuming that $G$ is semisimple.  We assume this from now on.  But we will see that the only aspect of semisimplicity that we will need is the definiteness of the Killing form.  We do not need the structure theory of semisimple Lie algebras.

For much of the material in this section I have been guided by the contents of \cite{BFR}.  In \cite{BFR} many possibilities are explored.  In contrast, we will here try to take the shortest path to what we need, and we will emphasize the extent to which the structures are canonical.  We will not examine what happens when we choose different $\mu_{\dia}$'s that have the same stability group.  But \cite{BFR} has considerable discussion of this aspect.

Let $K$ denote the $\Ad'$-stability subgroup of $\mu_{\dia}$, so that $x \mapsto \Ad'_x(\mu_{\dia})$ gives a $G$-equivariant diffeomorphism from $G/K$ onto the $\Ad'$-orbit of $\mu_{\dia}$.  We will usually work with $G/K$ rather than the orbit itself.

We let $\Kil$ denote the {\em negative} of the Killing form of $\fg$.  Then $\Kil$ is positive-definite because $G$ is compact.  The action $\Ad$ of $G$ on $\fg$ is by orthogonal operators with respect to $\Kil$, and the action $\ad$ of $\fg$ on $\fg$ is by skew-adjoint operators with respect to $\Kil$.  Because $\Kil$ is definite, there is a $Z_{\dia} \in \fg$ such that 
\begin{equation}
\label{eq1.1}
\mu_{\dia}(X) = \Kil(X,Z_{\dia}) \text{ for all } X \in \fg.
\end{equation}
It is easily seen that the $\Ad$-stability subgroup of $Z_{\dia}$ is again $K$.

Let $T_{\dia}$ be the closure in $G$ of the one-parameter group $r \mapsto \exp(rZ_{\dia})$, so that $T_{\dia}$ is a torus subgroup of $G$.  Then it is easily seen that $K$ consists exactly of all the elements of $G$ that commute with all the elements of $T_{\dia}$.  Note that $T_{\dia}$ is contained in the center of $K$ (but need not coincide with the center).  Since each element of $K$ will lie in a torus subgroup of $G$ that contains $T_{\dia}$, it follows that $K$ is the union of the tori that it contains, and so $K$ is connected (corollary 4.22 of  \cite{Knp}).  Thus for most purposes we can just work with the Lie algebra, $\fk$, of $K$ when convenient.  In particular, $\fk = \{X \in \fg: [X,Z_{\dia}]  = 0\}$, and $\fk$ contains the Lie algebra, $\ft_{\dia}$, of $T_{\dia}$.

Let $\fm = \fk^{\perp}$ with respect to $\Kil$.  Since $\Ad$ preserves $\Kil$, we see that $\fm$ is carried into itself by the restriction of $\Ad$ to $K$.  Thus $[\fk,\fm] \subseteq \fm$.  It is well-known, and explained in \cite{R22}, that $\fm$ can be conveniently identified with the tangent space to $G/K$ at the coset $K$ (which corresponds to the point $\mu_{\dia}$ of the coadjoint orbit).  We will review this in the next section.  Here we concentrate on the structures on $\fm$ that will give the K\"ahler structure on $G/K$.

The K\"ahler structure includes a symplectic form $\o_{\dia}$.  This is the Kirillov--Kostant--Souriau form, defined initially on $\fg$ by
\begin{equation}
\label{eq1.2}
\o_{\dia}(X,Y) = \mu_{\dia}([X,Y]) = \Kil([X,Y],Z_{\dia}) = \Kil(Y,[Z_{\dia}, X]).
\end{equation}
Because $Z_{\dia}$ is in the center of $\fk$, we see that if $X \in \fk$ then $\o_{\dia}(X,Y)  = 0$ for all $Y \in \fg$.  Conversely, if $X \in \fg$ and if $\o_{\dia}(X,Y) = 0$ for all $Y \in \fg$, then, because $\Kil$ is non-degenerate, we have $[X,Z_{\dia}] = 0$, so that $X \in \fk$.  Thus $\o_{\dia}$ ``lives'' on $\fm$ and is non-degenerate there.  Because $\Ad$ preserves $\Kil$ and $K$ stabilizes $Z_{\dia}$, it is easily seen that the restriction of $\Ad$ to $K$ preserves $\o_{\dia}$, that is,
\[
\o_{\dia}(\Ad_s(X),\Ad_s(Y)) = \o_{\dia}(X,Y)
\]
for all $X,Y \in \fm$ and $s \in K$.

We now follow the proof of proposition~12.3 of \cite{CdS} in order to construct a complex structure on $\fm$.  (I am grateful to Xiang Tang for bringing this proposition to my attention.  My original, somewhat longer, approach at this point was to begin working in the complexification of $\fg$ and $\fm$, as done in\cite{BFR}.)  See also the proof of theorem 1.36 of \cite{Brn} and the middle of the second proof of proposition~2.48i of \cite{McS}.  Because $\Kil$ is non-degenerate, there is a unique linear operator, $\G_{\dia}$, on $\fm$ such that
\begin{equation}
\label{eq1.3}
\o_{\dia}(X,Y) = \Kil(\G_{\dia}X,Y)
\end{equation}
for all $X,Y \in \fm$.  From equation \ref{eq1.2} we see that $\G_\dia$ is $\ad_{\zd}$ restricted to $\fm$, and so $\G_{\dia}$  is skew-symmetric, that is, $\G_{\dia}^* = -\G_{\dia}$.  Because $\o_{\dia}$ is non-degenerate, $\G_{\dia}$ is invertible.  Because $Z_{\dia}$ is in the center of $\fk$,  the $\Ad$-action of $K$ commutes with $\G_{\dia}$. Let $\G_{\dia} = |\G_{\dia}|J_{\dia}$ be the polar decomposition of $\G_{\dia}$.  Since $\G_{\dia}$ is invertible, so are $|\G_{\dia}|$ and $J_{\dia}$, and thus $J_{\dia}$ is an orthogonal transformation with respect to $\Kil$.  Because $\G_{\dia}$ is skew-symmetric, so is $J_{\dia}$, so that $J_{\dia}^{-1} = J_{\dia}^* = -J_{\dia}$, and $J_{\dia}$ commutes with $|\G_{\dia}|$.  In particular, $J_{\dia}^2 = -I$, where $I$ denotes the identity operator on $\fm$.  This means exactly that $J_{\dia}$ is a complex structure on $\fm$, preserved by the $\Ad$-action of $K$.

The final piece of structure is a corresponding inner product, $g_{\dia}$, on $\fm$, defined by
\[
g_{\dia}(X,Y) = \o_{\dia}(X,J_{\dia}Y) = \Kil(\G_{\dia}X,J_{\dia}Y) = \Kil(|\G_{\dia}|X,Y).
\]
Clearly $g_{\dia}$ is positive-definite, and is preserved by the $\Ad$-action of $K$.  It is $g_{\dia}$ that will give the Riemannian metric whose Dirac operator we will construct.  The complex structure $J_{\dia}$ will enable us to avoid the use of $\spin^c$ groups when constructing the Dirac operator.

But first we need to obtain a reasonably explicit expression for the Levi--Civita connection for the Riemannian metric corresponding to $g_{\dia}$.  For this purpose we need to examine the $\Ad$-action of $T_{\dia}$ on $\fm$.  By means of $J_{\dia}$ we make $\fm$ into a $\bC$-vector-space, by defining $iX$ to be just $J_{\dia}X$ for $X \in \fm$.  When we view $\fm$ as a $\bC$-vector-space in this way we will denote it by $\fm_{J_{\dia}}$.  Since the $\Ad$-action of $K$ (and thus of $T_{\dia}$) on $\fm$ commutes with $J_{\dia}$, this action respects the $\bC$-vector-space structure.  We define a $\bC$-sesquilinear inner product, $\Kil_{\dia}^{\bC}$, on $\fm$ by
\[
\Kil_{\dia}^{\bC}(X,Y) = \Kil(X,Y) + i\Kil(J_{\dia}X,Y).
\]
It is linear in the second variable.  (We follow the conventions in definition~5.6 of \cite{GVF}.) The $\Ad$-action of $K$ on $\fm_{J_{\dia}}$ is unitary for this inner product.  The $\Ad$-action of $T_{\dia}$ on $\fm_{J_{\dia}}$ then decomposes into a direct sum of one-dimensional complex representations of $T_{\dia}$, whose corresponding representations of $\ft_{\dia}$ are given by real-linear functions on $\ft_{\dia}$ whose values are pure-imaginary (the ``weights'' of the $\ad$-action).  We let $\D_\dia$ be the set of real-valued linear functionals $\a$ on $\ft_{\dia}$ such that $i\a$ is a weight of the $\ad$-action.  It will be convenient for us to set, for each real-linear real-valued functional $\a$ on $\ft_\dia$,
\[
\fm_{\a} = \{X \in \fm_{J_{\dia}}: \ad_Z(X) = i\a(Z)X = \a(Z)J_\dia X \ \  \text{ for all } Z \in \ft_{\dia}\}.
\]
Thus $\fm_{\a} = \{0\}$ exactly when $\a \notin \D_{\dia}$.  For any $X \in \fm_{\a}$ and $Y \in \fm_{J_{\dia}}$ we see from equation \ref{eq1.2} that
\[
\begin{aligned}
g_{\dia}(X,Y) &=  \o_{\dia}(X,J_{\dia}Y)  = \Kil([\zd, X], J_{\dia}Y)  \\
& =\Kil(\a(\zd)\jd X, \jd Y) = \a(Z_{\dia})\Kil(X,Y).
\end{aligned}
\]
Thus for $\a \in \D_{\dia}$ and $X \in \fm_{\a}$ with $X \ne 0$ we have
\[
0 < g_{\dia}(X,X) = \a(Z_{\dia})\Kil(X,X),
\]
and so $\a(Z_{\dia}) > 0$.  Thus in terms of the above notation we see that we obtain the following attractive description of $|\G_{\dia}|$:

\setcounter{theorem}{3}
\begin{proposition}
\label{prop1.4}
For each $\a \in \D_{\dia}$ the restriction of $|\G_{\dia}|$ to $\fm_{\a}$ is $\a(Z_{\dia})I_{\fm_{\a}}$, where $I_{\fm_{\a}}$ is the identity operator on $\fm_{\a}$.  In particular, $\a(\zd) > 0$, and
on $\fm_{\a}$ we have $g_{\dia} = \a(Z_{\dia})\Kil$.  If $P_{\a}$ denotes the orthogonal projection of $\fm$ onto $\fm_{\a}$, then $|\G_{\dia}| = \sum_{\a \in \D_{\dia}} \a(Z_{\dia})P_{\a}$.
\end{proposition}

Note that this proposition shows how strongly dependent $g_{\dia}$ is on the choice of $\mu_{\dia}$.  In contrast, different $\mu_{\dia}$'s that give $Z_{\dia}$'s that generate the same group $T_{\dia}$ may have the same subspaces $\fm_{\a}$.


\section{Levi--Civita connections \\ for invariant Riemannian metrics on $G/K$}
\label{sec2}

In this section we assume as before that $G$ is a connected compact semisimple Lie group, but we only assume that $K$ is a closed subgroup of $G$, not necessarily connected.  We will assume that we have an inner-product, $g_0$, on $\fm$ that is invariant under the $\Ad$-action of $K$.  We do not assume that $g_0$ is the restriction to $\fm$ of an $\Ad$-invariant inner product on $\fg$, as was assumed in \cite{R22}.  We will see shortly that much as in \cite{R22}, $g_0$ determines a $G$-invariant Riemannian metric on $G/K$.  We seek a formula for the Levi--Civita connection for this metric.  On $\fm$ there is a positive (for $\Kil$) invertible operator, $S$, such that $g_0(X,Y) = \Kil(SX,Y)$ for all $X,Y \in \fm$.  (So $S$ for a coadjoint orbit is the $|\G_0|$ of the previous section.) Note that $S$ commutes with the $\Ad$-action of $K$. Our formula will be expressed in terms of $S$.  In Section~\ref{sec3} we will use this formula to obtain a more precise formula for the Levi--Civita connection for a coadjoint orbit.  Toward the end of this section we will also discuss the divergence theorem for vector fields on $G/K$.  We need this for our discussion of the formal self-adjointness of Dirac operators in Section~\ref{sec7}.

As in \cite{R22}, we work with the module of tangent vector fields.  For brevity we will at times refer to such ``induced'' modules as ``bundles''.  We recall the setting here.  We let $A = C_{\bR}^{\infty}(G/K)$, which we often view as a subalgebra of $C_{\bR}^{\infty}(G)$.  The tangent bundle of $G/K$ is
\[
\cT(G/K) = \{V \in C^{\infty}(G,\fm): V(xs) = \Ad_s^{-1}(V(x)) \text{ for } x \in G, s \in K\}.
\]
It is an $A$-module for the pointwise product, and $G$ acts on it by translation.  We denote this translation action by $\l$.  Each $V \in \cT(G/K)$ determines a derivation, $\d_V$, of $A$ by
\[
(\d_Vf)(x) = D_0^t(f(x \exp(tV(x))),
\]
where $D_0^t$ means ``derivative in $t$ at $t = 0$''.  On $\cT(G/K)$ we have the canonical connection, $\nabla^c$, defined by
\begin{equation}
\label{eq2.1}
(\nabla_V^c(W))(x) = D_0^t(W(x \exp(tV(x)))
\end{equation}
for $V, \ W \in \cT(G/K)$.  It is not in general torsion-free.  Associated to it is the ``natural torsion-free'' \cite{KbN} connection, $\nabla^{ct}$, that is given (e.g. in theorem~6.1 of \cite{R22}) by
\[
\nabla^{ct} = \nabla^c + L^{ct},
\]
where $L_V^{ct}$ for any $V \in \cT(G/K)$ is the $A$-module endomorphism of $\cT(G/K)$ defined by
\begin{equation}
\label{eq2.2}
(L^{ct}_V W)(x) = (1/2)P[V(x),W(x)],
\end{equation}
where $P$ is the projection of $\fg$ onto $\fm$ along $\fk$.  Then $\nabla^{ct}$ is the Levi--Civita connection for the case in which $g_0$ is the restriction of $\Kil$ to $\fm$.  Both $\nabla^c$ and $\nabla^{ct}$ are $G$-invariant in the sense suitable for connections \cite{R22}.

Our given inner product $g_0$ determines a Riemannian metric on $G/K$, also denoted by $g_0$, defined by
\[
(g_0(V,W))(x) = g_0(V(x),W(x))
\]
for all $V,W \in \cT(G/K)$ and $x \in G$.  Thus $g_0(V,W) \in A$. When there is no ambiguity about the choice of $g_0$ we will often write $\<V, W\>_A$ instead of $g_0(V,W) \in A$. This Riemannian metric is $G$-invariant (and every $G$-invariant Riemannian metric arises in this way).  We seek to adjust $\nabla^{ct}$ to obtain the Levi--Civita connection, $\nabla^0$, for $g_0$.  A convenient method for doing this is given by theorem~X.3.3 of \cite{KbN} (or equation~13.1 of \cite{Nmz}, where there is a sign error).  We seek $\nabla^0$ in the form $\nabla^{ct} + L^S$, where $L^S$ is an $A$-linear map from $\cT(G/K)$ into the $A$-endomorphisms of $\cT(G/K)$.  We require that $L^S$ be symmetric, that is that $L_W^SV = L_V^SW$ for all $V,W \in \cT(G)$, since this ensures that $\nabla^0$ is torsion free, because $\nabla^{ct}$ is.  As seen in \cite{R22}, by translation invariance we can calculate at $x = e$, the identity element of $G$.  Then according to theorem~X.3.3 of \cite{KbN} we are to determine the symmetric bilinear form $\Phi$ on $\fm$ that satisfies the equation
\begin{equation}
\label{eq2.3}
2g_0(\Phi(X,Y),Z) = g_0(X,P[Z,Y]) + g_0(P[Z,X],Y)
\end{equation}
for all $X,Y,Z \in \fm$. For the reader's convenience we recall the reasoning.  For $x = e$ we have $L_X^{ct}(Y) = (1/2)P[X,Y]$ for $X,Y \in \fm$.  Set $L^S$ on $\fm$ to be $L_X^S(Y) = \Phi(X,Y)$.  Then the above equation becomes
\[
g_0(L_X^SY,Z) = g_0(X,L_Z^{ct}Y) + g_0(L_Z^{ct}X,Y).
\]
When we add to this equation its cyclic permutation
\[
g_0(L_Z^SX,Y) = g_0(Z,L_Y^SX) + g_0(L_Y^S(Z),X)
\]
and use the symmetry of $g_0$ and $\Phi$ and the fact that $L_Z^{ct}Y = -L_Y^{ct}Z$, we obtain
\[
g_0(L^S_XY,Z) + g_0(Y,L_X^SZ) = -g_0(L_X^{ct}Y,Z) - g_0(Y,L_X^{ct}Z).
\]
This says exactly that the operator $L_X^{ct} + L_X^S$ on $\fm$ is skew-symmetric with respect to $g_0$.  This implies that when $L^S$ is extended to $\cT(G/K)$ by $G$-invariance (in the sense that $\l_x(L_V^SW) = L_{\l_xV}^S\l_xW$ as discussed in section~5 of \cite{R22}) the connection $\nabla^{ct} + L^S$ is compatible with the Riemannian metric $g_0$ (as seen, for example, from corollary~5.2 of \cite{R22}).  This connection is also torsion-free, and thus it is the Levi--Civita connection for $g_0$.

When we rewrite equation \ref{eq2.3} in terms of $\Kil$ and $S$ we obtain
\begin{align*}
2\Kil(S\Phi(X,Y),Z) &= \Kil(SX,P[Z,Y]) + \Kil(P[Z,X],SY) \\
&= \Kil([Y,SX],Z) + \Kil(Z,[X,SY]).
\end{align*}
Since this must hold for all $Z$, we see that
\[
L_X^SY = \Phi^0(X,Y) = (1/2)S^{-1}P([X,SY] + [Y,SX]).
\]
By $G$-invariance as above
\begin{equation}
\label{eq2.4}
(L_V^SW)(x) = (1/2)S^{-1}P([V(x),SW(x)] + [W(x),SV(x)])
\end{equation}
for $V,W \in \cT(G/K)$ and $x \in G$.  We thus obtain: 

\setcounter{theorem}{4}
\begin{theorem}
\label{th2.5}
The Levi--Civita connection for the Riemannian metric $g_0$ is $\nabla^0 = \nabla^{ct} + L^S$ where $L^S$ is defined by \eqref{eq2.4} and $S$ relates $g_0$ to $\Kil$ as above.
\end{theorem}

Let $\D$ denote the set of eigenvalues of $S$, and for each $\a \in \D$ let $\fm_{\a}$ denote the corresponding eigensubspace.  For $\a,\b \in \D$ and $X \in \fm_{\a}$, $Y \in \fm_{\b}$ we see that
\begin{equation}
\label{eq2.6}
L_X^SY = (1/2)S^{-1}P([X,\b Y] + [Y,\a X])= (1/2)(\b-\a)S^{-1}P[X,Y],
\end{equation}
and thus the complication in getting a more precise formula lies in expressing $S^{-1}P[X,Y]$ in terms of the eigensubspaces of $S$.  In Section~\ref{sec3} we will see how to obtain such a more precise formula for the case of coadjoint orbits.

But first we derive here a form of the divergence theorem for our vector fields, because we will need it in Section~\ref{sec7}, and equation \eqref{eq2.4} is important for its proof.  We recall from \cite{R22} that by a standard module frame for $\cT(G/K)$ with respect to the Riemannian metric $g_0$ we mean a finite collection $\{W_j\}$ of elements of $\cT(G/K)$ that have the reproducing property
\[
V = \sum W_j\<W_j, \ V\>_A
\]
for all $V \in \cT(G/K)$.  (We view $\cT(G/K)$ as a right $A$-module, following the conventions in \cite{GVF}.)

\setcounter{theorem}{6}
\begin{definition}
\label{def2.7}
Let $\nabla^0$ be the Levi--Civita connection for the Riemannian metric $g_0$ on $G/K$.  We define the {\em divergence}, $\div(V)$, of an element $V \in \cT(G/K)$, with respect to $g_0$, to be
\setcounter{equation}{7}
\begin{equation}
\label{eq2.8}
\div(V) = \sum_j g_0(\nabla_{W_j}^0V,W_j),
\end{equation}
where $\{W_j\}$ is a standard module frame for $\cT(G/K)$.
\end{definition}

It is not difficult to check that this definition coincides with the usual definition of the divergence in terms of differential forms, but we do not need this fact here.  We should make sure that our definition is independent of the choice of the frame $\{W_j\}$.  To prove our divergence theorem we actually need a slightly more general form of frames, so we give the independence argument in terms of these.  The argument is essentially well-known.

\setcounter{theorem}{8}
\begin{proposition}
\label{prop2.9}
Let $A$ be a commutative ring and let $E$ be an $A$-module that is equipped with an $A$-valued symmetric bilinear form $\<\cdot,\cdot\>_A$.  Assume that there exist biframes for $E$ with respect to this bilinear form, that is, there are finite sets $\{(W_j,{\tilde W}_j)\}$ of pairs of elements of $E$ such that $V = \sum W_j\<{\tilde W}_j,V\>_A$ for every $V \in E$.  Then for any $A$-bilinear form $\b$ on $E$, not necessarily symmetric, with values in some $A$-module, the sum $\sum_j \b(W_j,{\tilde W}_j)$ is independent of the choice of biframe.
\end{proposition}

\begin{proof}
Let $\{(U_k,{\tilde U}_k)\}$ be another biframe.  Then
\begin{align*}
\sum_j \b(W_j,{\tilde W}_j) &= \sum_j \sum_{k,l} \b(U_k\<{\tilde U}_k,W_j\>_A,\ U_l\<{\tilde U}_l,{\tilde W}_j\>_A) \\
&= \sum_{k,l} \b(U_k,U_l) \sum_j \<{\tilde U}_k,W_j\<{\tilde W}_j,{\tilde U}_l\>_A\>_A \\
&= \sum_k \b(U_k, \sum_lU_l\<{\tilde U}_l,{\tilde U}_k\>_A) = \sum_k \b(U_k,{\tilde U}_k).
\end{align*}
This proof can be made more conceptual by noting that $\<\cdot,\cdot\>_A$ establishes an isomorphism of $E\otimes_A E$ with $End_A(E)$.
\end{proof}

Our greater generality is needed because we want to use frames that involve the fundamental vector fields ${\hat X}$, for $X \in \fg$, that correspond to the action of $G$ by translation on $G/K$.  As shown in section~4 of \cite{R22}, they are given by
\[
{\hat X}(x) = -P\Ad_x^{-1}(X).
\]
It is also shown in section~4 of \cite{R22} that if $\{X_j\}$ is an orthonormal basis for $\fg$ for $\Kil$, then $\{{\hat X}_j\}$ is a standard module frame for the Riemannian metric on $G/K$ coming from restricting $\Kil$ to $\fm$.  Thus for any $V \in \cT(G/K)$ we have
\[
V = \sum_j {\hat X}_j\Kil({\hat X}_j,V) = \sum {\hat X}_j g_0(S^{-1}{\hat X}_j,V).
\]
From this we see that the collection $\{({\hat X}_j,S^{-1}{\hat X}_j)\}$ is a biframe for $\cT(G/K)$ when $\cT(G/K)$ is equipped with $g_0$. On $G/K$ we use the $G$-invariant measure coming from a choice of Haar measure on $G$.

\begin{theorem}
\label{th2.10}
Let $g_0$ be a $G$-invariant Riemannian metric on $G/K$ and let $\div(V)$ be defined as above for $g_0$.  Then for any $V \in \cT(G/K)$ we have
\[
\int_{G/K} \div(V) = 0.
\]
\end{theorem}

\begin{proof}
We have $\nabla^0 = \nabla^c + L^{ct} + L^S$. We split $\int_{G/K} \div(V)$ into the corresponding three terms, and show that each is 0. The first term is
\[
\int_{G/K}\sum_j g_0(\nabla^c_{W_j} V, \ W_j) .
\]
It is independent of the choice of frame $\{W_j\}$ by Proposition \ref{prop2.9}, and by that proposition we can, in fact, use the biframe defined just above. Now $\nabla^c$ is compatible with $g_0$, and so
\[
g_0(\nabla^c_{{\hat X}_j}V,S^{-1}{\hat X}_j) = \d_{{\hat X}_j}(g_0(V,S^{-1}{\hat X}_j)) - \sum g_0(V,\nabla^c_{{\hat X}_j}(S^{-1}{\hat X}_j)) .
\]
But as discussed in the proof of theorem~8.4 of \cite{R22}, for any $X \in \fg$ and any $f \in A$ we have $\int_{G/K} \d_{\hat X}(f) = 0$, because $\d_{\hat X}(f)$ is the uniform limit of the quotients $(\l_{\exp(-tX)}f - f)/t$ as $t \to 0$, and the integral of each of these quotients is $0$ by the $G$-invariance of the measure on $G/K$.  
Thus we see that we would like to show that
\[
\int_{G/K} g_0(V, \ \sum_j\nabla^c_{{\hat X}_j}(S^{-1}{\hat X}_j)) \ = \ 0.
\]
For that it suffices to show that 
\[
 \sum_j\nabla^c_{{\hat X}_j}(S^{-1}{\hat X}_j) \ = \ 0 .
 \]
 But $\nabla^c$ only involves derivatives, and since $S^{-1}$ is constant, it is clear from equation \ref{eq2.1} that $S^{-1}$ commutes with $\nabla^c$. It thus suffices to show that 
$ \sum \nabla_{{\hat X}_j}^c {\hat X}_j = 0$. This was shown at the end of the proof of theorem 8.4 in \cite{R22}. We recall the reasoning here. Early in section 6 of \cite{R22} it is shown that for each $X, Y \in \fg$
\[
(\nabla_{\hat X}^c{\hat Y})(x) = -P([P\Ad_x^{-1}(X),\Ad_x^{-1}(Y)])
\]
for all $x \in G$.  By Proposition~\ref{prop2.9} for each fixed $x \in G$ we can choose the basis $\{X_j\}$ such that  $\{\Ad_x^{-1}(X_j)\}$ is the union of a $\Kil$-orthonormal basis for $\fm$ and one for $\fk$.  For such a basis $(\nabla_{{\hat X}_j}^c {\hat X}_j)(x) = 0$ for each $j$.
 
Now let $L^0 = L^{ct} + L^S$. It remains to show that
\[
\int_{G/K} \sum_j g_0(L^0_{W_j} V, \ W_j) \ = \ 0
\]
for each $V \in \cT(G/K)$. But $\nabla^0 = \nabla^c + L^0$, and $\nabla^0$ is assumed to be compatible with $g_0$. Consequently each $L^0_U$ is skew-adjoint for $g_0$. Thus
\[
g_0(L^0_{W_j}V, \ W_j) = -g_0(V, \ L^0_{W_j} W_j) ,
\]   
and so we see that it suffices to show that $\sum_j L^0_{W_j} W_j = 0$. To show this we treat $L^{ct}$ and $L^S$ separately. Now from equation \ref{eq2.2} we see that for each $j$
\[
(L^{ct}_{W_j} W_j)(x) = (1/2)P[W_j(x), \ W_j(x)] = 0 .
\]
Thus  $\sum_j L^{ct}_{W_j} W_j = 0$

Finally, from equation \ref{eq2.4} we see that
\begin{eqnarray*}
(\sum_j L^S_{W_j}W_j)(x) &=& (1/2)S^{-1}P\sum_j [W_j(x), SW_j(x)] + [W_j(x), SW_j(x)] \\
&=& S^{-1}P\sum_j [W_j(x), \ SW_j(x)] .
\end{eqnarray*}
But $\{W_j(x)\}$ is a frame for $\fm$ and $g_0$, for each $x \in G$, and by Proposition \ref{prop2.9} the above expression is independent of the chosen frame. Notice that $S$ is positive for $g_0$ as well as for $\Kil$. Consequently as frame we can choose a $g_0$-orthonormal basis for $\fm$ consisting of eigenvectors of $S$. It is then clear that $(\sum_j L^S_{W_j}W_j)(x) = 0$.
\end{proof}


\section{The Levi--Civita connection for coadjoint orbits}
\label{sec3}

We now return to the setting of coadjoint orbits as in Section~\ref{sec1}, with $S = |\G_{\dia}|$. 
We will obtain here the more precise formula for $|\G_\dia |^{-1} P[X, \ Y]$ that Theorem
\ref{th2.5} indicates we need in order to obtain a precise formula for the Levi-Civita connection
for $g_\dia$. Motivation for some of the expressions that we consider can be found by working in the complexification of $\fg$ along the lines used in \cite{BFR}. For any $\a,  \b \in \D_\dia$ set
$|\a - \b| = \a - \b$ if $(\a - \b)(Z_\dia) \geq 0$, and otherwise set $|\a - \b| = \b - \a$, so
that always $|\a - \b|(Z_\dia) \geq 0$. Of course $|\a - \b|$ may not be in $\D_\dia$. Recall
from Proposition \ref{prop1.4} that if $\g \in \D_\dia$ then $\g(\zd) > 0$. We do not
have $[\fm, \fm] \subseteq \fm$, but nevertheless:

\begin{lemma}
\label{lem3.1}
Let $\a, \b \in \D_{\dia}$, and let $X \in \fm_\a$ and $Y \in \fm_\b$. Then
\[
[X, Y] - [\jd X, \jd Y] \ \in \ \fm_{\a + \b}  \quad \text{(so \ $= 0$ \ if \ $\a+\b \notin \D_{\dia}$)},
\]
while
\[
[X, \ Y] + [J_{\dia}X, \ J_{\dia}Y] \in \begin{cases}
 \fk \quad &\text{if $\a = \b$} \\
 \fm_{|\a-\b|} \quad &\text{if $\a \ne \b$ \quad (so \ $= 0$ \ if \ $|\a-\b| \notin \D_{\dia}$).}
\end{cases}
\]
Thus, on adding, we find that 
\[
[X, \ Y]  \in \begin{cases}
 \fm_{\a + \b} \oplus  \fm_{|\a - \b|} \quad &\text{if $\a \ne \b$} \\
 \fm_{2\a} \oplus \fk \quad &\text{if $\a = \b$  .}
\end{cases}
\]
Furthermore,
\[
\jd([X, Y] - [\jd X, \jd Y]) = [\jd X, Y] + [X, \jd Y] \ ,
\]
while if $\a \ne \b$ then
\[
\jd([X, Y] + [\jd X, \jd Y]) = \mathrm{sign}(\a(\zd) - \b(\zd))([\jd X, Y] - [X, \jd Y]) \ .
\]
If $\a = \b$ then $\jd P ([X, Y] + [\jd X, \jd Y]) = 0$.
\end{lemma}

\begin{proof} Note that $[X, Y]$ need not be in $\fm$. Let $Z \in \ft$. Within the calculations
below we will, for brevity, often write just $\a$ for $\a(Z)$ and similarly for $\b$. Then
from the Jacobi identity we have
\[  
\begin{aligned}
\ad_Z([X, Y] &- [\jd X, \jd Y])  \\
&= [\a\jd X, Y] + [X, \b\jd Y] + [\a X, \jd Y] + [\jd X, \b Y] \\
&= (\a + \b)([\jd X, Y] + [X, \jd Y]) \ .
\end{aligned}
\]
On substituting $\jd X$ for $X$ in the equation above we obtain
\[
\ad_Z([\jd X, Y] + [ X, \jd Y]) = -(\a + \b)([X, Y] - [\jd X, \jd Y]) \ ,
\]
and on combining these two equations we obtain
\[
(\ad_Z)^2([X, Y] - [\jd X, \jd Y]) = -(\a + \b)^2([X, Y] - [\jd X, \jd Y]) \ .
\]

Recall that $\ad_Z$ carries $\fm$ into itself and sends 
$\fk$ to $\{0\}$, so the range of $\ad_Z$ is in $\fm$. Now let $Z = \zd$,
so that $\a > 0$ and $\b > 0$. Then we see from the above calculations that
$([X, Y] - [\jd X, \jd Y]) \in \fm$. Recall also that $\ad_{\zd} = |\G_\dia|\jd$, so that 
$(\ad_{\zd})^2 = -|\G_\dia|^2$. Then from the above calculations it becomes
clear that $[X, Y] - [\jd X, \jd Y] \ \in \ \fm_{\a + \b}$. Of course it may be that
$\a + \b \notin \D_\dia$ so that $\fm _{\a+\b} = \{0\}$. From the above calculations
we see furthermore that
\[
\jd([X, Y] - [\jd X, \jd Y]) = [\jd X, Y] + [X, \jd Y]  \ .
\]

In the same way, for any $Z \in \ft$ we have
\[
\ad_Z([X, Y] + [\jd X, \jd Y]) = (\a - \b)([\jd X, Y] - [X, \jd Y]) \ 
\]
and
\[
\ad_Z([\jd X, Y] - [X, \jd Y]) = (\b - \a)([X, Y] + [\jd X, \jd Y]) \ ,
\]
so that
\[
(\ad_Z)^2([X, Y] + [\jd X, \jd Y]) = -(\a - \b)^2([X, Y] + [\jd X, \jd Y]) \ .
\]
If $\a = \b$ then it is clear from these calculations that
$([X, Y] + [\jd X, \jd Y]) \in \fk$. If $\a \ne \b$, then on letting $Z = \zd$
and arguing as above, we see that $([X, Y] + [\jd X, \jd Y]) \in \fm_{|\a - \b|}$
for the definition of ${|\a - \b|}$ given above. The statement about 
$\jd([X, Y] + [\jd X, \jd Y])$ now follows much as before.
\end{proof}

Recall the definition of $L^S$ from equations 2.4 and 2.5. We now use the above lemma to
obtain a more precise formula for $L^S$ for the present case in which $S = \G_\dia$.
We denote this $L^S$ for $\G_\dia$ by $L^\dia$. 

\begin{proposition}
\label{prop3.2}
Let $\a, \b \in \D_\dia$, and let $X \in \fm_\a$ and $Y \in \fm_\b$. Then
\[
\begin{aligned}
4L^\dia_X Y & \ = \
(\b(\zd) - \a(\zd))(\a(\zd) + \b(\zd))^{-1}([X,Y] - [\jd X, \jd Y]) \\
& \quad \quad + \  \mathrm{sign}(\b(\zd) - \a(\zd))([X,Y] + [\jd X, \jd Y]) \ ,
\end{aligned}
\]
as long as we make the convention that $\mathrm{sign}(0) = 0$.
\end{proposition}

\begin{proof}
From Lemma \ref{lem3.1} we see that
\[
\begin{aligned}
|\G_\dia|^{-1} & P[X, Y]  \\
&=
(1/2)|\G_\dia|^{-1} P(([X,Y] - [\jd X, \jd Y]) + ([X,Y] + [\jd X, \jd Y]))  \\
& = (1/2)((\a(\zd) + \b(\zd))^{-1}([X,Y] - [\jd X, \jd Y])  \\
& \quad + |\a(\zd) - \b(\zd)|^{-1}P([X,Y] + [\jd X, \jd Y]) ,
\end{aligned}
\]
where the last term must be taken to be 0 if $\a = \b$.
On substituting this into equation 2.5 and simplifying, 
we obtain the desired expression for $L^\dia_X(Y)$.
\end{proof}

Recall now that the Levi-Civita connection for $g_\dia$ is
$\nabla^\dia = \nabla^{ct} + L^\dia   = \nabla^c + L^{ct} + L^\dia$,
where on $\fm$ we have 
$L^{ct}_X Y = (1/2)P[X, Y]$. If we set $L^{\dia t} = L^{ct} + L^\dia$, then from 
Proposition 3.2 we see that on $\fm$ we have
\[
\begin{aligned}
4L^{\dia t}_X  &Y  \\
=& \ (1 +(\b(\zd) - \a(\zd))(\a(\zd) + \b(\zd))^{-1})([X,Y] - [\jd X, \jd Y]) \\
& + \ (1 + \mathrm{sign}(\b(\zd) - \a(\zd)))P([X,Y] + [\jd X, \jd Y]) \ ,  \\
=& \ 2\b(\zd)(\a(\zd) + \b(\zd))^{-1}([X,Y] - [\jd X, \jd Y]) \\
& + \ (1 + \mathrm{sign}(\b(\zd) - \a(\zd)))P([X,Y] + [\jd X, \jd Y]) \ .
\end{aligned}
\]
When we extend this to $\cT(G/K)$ by $G$-invariance, and let $\cT^\a(G/K)$ denote the subspace of $\cT(G/K)$ consisting of elements whose range is in $\fm_\a$, we obtain:

\begin{theorem}
\label{thm3.3}
The Levi-Civita connection $\nabla^\dia$ for the Riemannian metric $g_\dia$
is given for $V \in \cT^\a(G/K)$ and $W \in \cT^\b(G/K)$, for $\a, \b \in \D_\dia$, by 
\[
\begin{aligned}
(\nabla^\dia_V & W)(x)   =  (\nabla^c_V W)(x) \\
&  +  \ (1/4)\big(2\b(\zd)(\a(\zd) + \b(\zd))^{-1}([V(x),W(x)] - [\jd V(x), \jd W(x)]) \\
& + \ (1 + \mathrm{sign}(\b(\zd) - \a(\zd)))P([V(x),W(x)] + [\jd V(x), \jd W(x)])\big) \ 
\end{aligned}
\]
for all $x \in G$.
\end{theorem}

The above formula should be compared to formula 7.15 in \cite{BIL}. We remark that from Theorem 3.3 and Lemma 3.1 it is easily seen that  $g_\dia$ is ``naturally reductive'' \cite{KbN, Agr}, so has Levi-Civita connection equal to $\nabla^{ct}$ \cite{Agr}, exactly when $G/K$ is a symmetric space, that is, when $[\fm, \fm]\subseteq \fk$

In our K\"ahler situation we expect that $\jd$ will commute with $\nabla^\dia$. This is essential
for the construction that we will give shortly for the Dirac operator for $g_\dia$. We now check
this fact directly.

\begin{proposition}
\label{prop3.4}
With notation as above, $\jd$ commutes with $\nabla^\dia$.
\end{proposition}

\begin{proof}
It is easily seen that $\jd$ commutes with $\nabla^c$, so we only need to show that it commutes
with $L^{\dia t}$. Note that in general $\jd$ does not commute with $L^{ct}$, so we need to work
with the combination $L^{ct} + L^\dia = L^{\dia t}$. By $G$-invariance it suffices to deal just with
elements of $\fm$. Let $\a, \b \in \D_\dia$, and let $X \in \fm_\a$ and $Y \in \fm_\b$. For 
brevity we again often write just $\a$ for $\a(\zd)$ and similarly for $\b$ within our calculations. Then
when we apply $\jd$ to $L^{\dia t}$ and apply the results of
Lemma \ref{lem3.1}, we obtain
\[
\begin{aligned}
4\jd L^{\dia t}_X & Y  
= \ 2\b(\a + \b)^{-1}([\jd X,Y] + [ X, \jd Y]) \\
& + \ (1 + \mathrm{sign}(\b - \a))\mathrm{sign}(\a - \b)P([\jd X,Y] - [ X, \jd Y]) \ ,
\end{aligned}
\]
while
\[
\begin{aligned}
4 L^{\dia t}_X & (\jd Y)  
= \ 2\b(\a + \b)^{-1}([\jd X,Y] + [ X, \jd Y]) \\
& + \ (1 + \mathrm{sign}(\b - \a))P([ X, \jd Y] - [\jd X,  Y])  \ .
\end{aligned}
\]
Notice that
\[ 
\begin{aligned}
( & 1 + \mathrm{sign}(\b - \a))\mathrm{sign}(\a - \b)P([\jd X,Y] - [ X, \jd Y]) \\
& = \ (1 + \mathrm{sign}(\b - \a))\mathrm{sign}(\b - \a)P( [ X, \jd Y] - [\jd X,Y] ) \\
& = \ (\mathrm{sign}(\b - \a) + 1)P( [ X, \jd Y] - [\jd X,Y] ) \ .
\end{aligned} 
\]
Thus $\jd L^{\dia t}_X Y = L^{\dia t}_X (\jd Y)$ as desired.
\end{proof}


\section{The spinor representation}
\label{sec5}

In view of the results of the previous sections, it is appropriate to consider in general an even-dimensional real vector space $\fm$ with a given inner product $g_0$, a compact Lie group $K$ that is not required to be semisimple or connected, a representation $\pi$ (instead of $\Ad|_K$) of $K$ on $\fm$ preserving $g_0$, and a complex structure $J$ on $\fm$ respecting both $g_0$ and $\pi$.  For use in constructing a Dirac operator we seek a representation of the complex Clifford algebra over $\fm$ for $g_0$ that respects the action of $K$.  (Many coadjoint orbits are not $\spin$ manifolds \cite{BIL, Owc, DlN, LwM}, only $\spin^c$.)  Much of the material in this section is taken from chapter~5 of \cite{GVF}.  The exposition in \cite{GVF} is especially suitable for our needs, and it includes much detail on a number of aspects.  But as before, here we will try to take the shortest path to what we need.  An important point is that we will find that because of the complex structure we do not need to involve the $\spin^c$ groups, with their attendant complexities.

As in \cite{LwM,GVF}, we will denote the complex Clifford algebra over $\fm$ for $g_0$ by $\bC\ell(\fm)$.  It is the complexification of the real Clifford algebra for $\fm$ and $g_0$.  We follow the convention that the defining relation is
\[
XY + YX = -2g_0(X,Y)1.
\]
We include the minus sign for consistency with \cite{LwM,R22}.  Thus in applying the results of the first pages of chapter~5 of \cite{GVF} we must let the $g$ there to be $-g_0$.  After exercise 5.6 of \cite{GVF} it is assumed that $g$ is positive, so small changes are needed when we use the later results in \cite{GVF} but with our different convention.  The consequence of including the minus sign is that in the representations which we will construct the elements of $\fm$ will act as skew-adjoint operators, just as they do for orthogonal or unitary representations of $G$ if $\fm$ arises as in the previous section, rather than as self-adjoint operators as happens when the minus sign is omitted.

Because $\fm$ is of even dimension, the algebra $\bC\ell(\fm)$ is isomorphic to a full matrix algebra \cite{GVF}.  We equip $\bC\ell(\fm)$ with the involution $^*$ (conjugate linear, with $(ab)^* = b^*a^*$) that takes $X$ to $-X$ for $X \in \fm$ (again so that the elements of $\fm$ are skew-adjoint).

Let $O(\fm,g_0)$ denote the group of operators on $\fm$ orthogonal for $g_0$.  By the universal property of Clifford algebras each element $R$ of $O(\fm,g_0)$ determines an automorphism of $\bC\ell(\fm)$ (a ``Bogoliubov'' automorphism) given on a product $X_1\cdots X_p$ of elements of $\fm$ in $\bC\ell(\fm)$ by 
\begin{equation}
\label{eq5.1}
R(X_1X_2 \cdots X_p) = R(X_1)R(X_2) \cdots R(X_p).
\end{equation}
In this way we obtain a homomorphism from $O(\fm,g_0)$ into the automorphism group of $\bC\ell(\fm)$.  Since $\pi$ gives a homomorphism of $K$ into $O(\fm,g_0)$ we obtain a homomorphism, still denoted by $\pi$, of $K$ into the automorphism group of $\bC\ell(\fm)$, which extends the action of $K$ on $\fm$.  The Lie algebra $so(\fm,g_0)$ of $O(\fm,g_0)$ will then act as a Lie algebra of derivations of $\bC\ell(\fm)$, given for $L \in so(\fm,g_0)$ by
\begin{equation}
\label{eq5.2}
\begin{aligned}
L(X_1X_2\cdots X_p) &= L(X_1)X_2\cdots X_p \\
&+ \ X_1L(X_2)X_3\cdots X_p \ + \ \cdots + X_1 \cdots X_{p-1}L(X_p).
\end{aligned}
\end{equation}
Corresponding to this we have an action of $\fk$ as derivations of $\bC\ell(\fm)$, again
denoted by $\pi$.

Because $\bC\ell(\fm)$ is isomorphic to a full matrix algebra, it has, up to equivalence, exactly one irreducible representation.  We seek an explicit construction of such a representation in a form which makes manifest that this representation carries an action of $K$ that is compatible with the action of $K$ on $\bC\ell(\fm)$.  As shown in \cite{GVF} beginning with definition~5.6, the complex structure $J$ on $\fm$ leads to an explicit construction.  (See also the discussion after corollary~5.17 of \cite{LwM}.)  We will denote the resulting Hilbert space for this representation by $\cS$, for ``spinors''.

To begin with, we use $J$ to view $\fm$ as a complex vector space by setting $iX = JX$, as we did earlier.  We then define a positive-definite sesquilinear form, i.e., complex inner product, on $\fm$, by
\[
\<X,Y\>_J = g_0(X,Y) + ig_0(J(X),Y).
\]
Note that, as in \cite{GVF}, we take it linear in the second variable.  When we view $\fm$ as a complex vector space with this inner product, we denote it by $\fm_J$.  We note that because $\pi$ commutes with $J$ and preserves $g_0$, it is a unitary representation of $K$ on $\fm_J$ (so that, in particular, actually $\pi(K) \subseteq SO(\fm,g_0)$).  As in definition~5.7 of \cite{GVF} we let $\cF(\fm_J)$ denote the complex exterior algebra $\bigwedge^*{\fm_J}$ over $\fm_J$.  It is referred to in \cite{GVF} as the (unpolarized) Fock space.  It will be our space $\cS$ of spinors, and we will write $\cF(\fm_J)$ or $\cS$ as convenient.  Then we equip $\cS$ with the inner product determined by
\begin{equation}
\label{eq5.3}
\<X_1\wedge \dots \wedge X_p, \ Y_1 \wedge \dots \wedge Y_q\>_J = \d_{pq} \det[\<X_k,Y_l\>_J],
\end{equation}
which is equation 5.17a of \cite{GVF}.  Let $U(\fm_J)$ denote the unitary group of $\fm_J$.  By the universal property of exterior algebras the action of $U(\fm_J)$ on $\fm_J$ extends to an action on $\cF(\fm_J)$ by exterior-algebra automorphisms, defined in much the same way as in equation \eqref{eq5.1}.  By means of the homomorphism $\pi$ from $K$ into $U(\fm_J)$  we obtain an action of $K$ as automorphism of $\cF(\fm_J)$, again denoted by $\pi$.  Then the Lie algebra $u(\fm_J)$ of $U(\fm_J)$ will act as a Lie algebra of exterior-algebra derivations of $\cF(\fm_J)$, and by this means we obtain an action, $\pi$, of $\fk$ as derivations of $\cF(\fm_J)$.

We need a representation of $\bC\ell(\fm)$ on $\cS$.  As done shortly after exercise 5.12 of \cite{GVF}, we define annihilation and creation operators, $a_J(X)$ and $a_J^{\dag}(X)$, on $\cF(\fm_J)$ for $X \in \fm$ by
\[
a_J(X)(X_1 \wedge \dots \wedge X_p) = \sum_{j=1}^p (-1)^{j-1}\<X,X_j\>_J \ X_1 \wedge \dots \wedge {\hat X}_j \wedge \dots \wedge X_p
\]
(where ${\hat X}_j$ means to omit that term), and
\[
a_J^{\dag}(X)(X_1 \wedge \dots \wedge X_p) = X \wedge X_1 \wedge \dots \wedge X_p
\]
for $X_1,\dots,X_p \in \fm_J$.  Note that $a_J(X)$ is conjugate linear in $X$.  One then checks, much as done in the paragraph before definition 5.1 of \cite{GVF}, that
\[
a_J(X)a_J^{\dag}(Y) = a_J^{\dag}(Y)a_J(X) = \<X,Y\>_JI_{\cS},
\]
where $I_{\cS}$ is the identity operator on $\cF(\fm_J)$, and
\[
a_J(X)a_J(Y) + a_J(Y)a_J(X) = 0 = a_J^{\dag}(X)a_J^{\dag}(Y) + a_J^{\dag}(Y)a_J^{\dag}(X)
\]
for $X,Y \in \fm$.  We then set
\[
\k_J(X) = i(a_J(X) + a_J^{\dag}(X)).
\]
(So the $i$ here reflects our sign convention, different from that of \cite{GVF}.)  Note that $\k_J(X)$ is only real-linear in $X$.  Using the anti-commutation relations above, we see that
\[
\k_J(X)\k_J(Y) + \k_J(Y)\k_J(X) = -\<X,Y\>_J - \<Y,X\>_J = -2g_0(X,Y)
\]
for $X,Y \in \fm$, where we omit $I_{\cS}$ on the right as is traditional.  But this is the relation that defines $\bC\ell(\fm)$.  Thus $\k_J$ extends by universality to give a homomorphism, again denoted by $\k_J$, from $\bC\ell(\fm)$ into the algebra, $\cL(\cF(\fm_J))$, of linear operators on $\cF(\fm_J)$.  Let $\dim_{\bR}(\fm) = 2n$. Then $\dim_{\bC}(\fm_J) = n$, so that $\dim_\bC(\cF(\fm_J)) = 2^n$.  But $\dim_{\bC}(\bC\ell(\fm)) = 2^{2n} = (2^n)^2$.  Since $\bC\ell(\fm)$ is isomorphic to a full matrix algebra, and since $\k_J$ is clearly not the $0$ homomorphism, the homomorphism $\k_J$ must be bijective, and gives an irreducible representation of $\bC\ell(\fm)$ on $\cF(\fm_J)$.  Thus we can take $\cS = \cF(\fm_J)$ as our Hilbert space of spinors, with the action of $\bC\ell(\fm)$ on $\cS$ given by $\k_J$.

Recall that we have actions of $O(\fm)$ on $\bC\ell(\fm)$ and of $U(\fm_J)$ on $\cF(\fm_J)$.  Since $U(\fm_J) \subset SO(\fm)$, we have an action of $U(\fm_J)$ on $\bC\ell(\fm)$.  Let us denote by $\rho$ the actions of $U(\fm_J)$ on both $\bC\ell(\fm)$ and $\cF(\fm_J)$.  A crucial fact for us is:

\setcounter{theorem}{3}
\begin{proposition}
\label{prop5.4}
The action $\k_J$ of $\bC\ell(\fm)$ on $\cF(\fm_J)$ respects the actions $\rho$ of $U(\fm_J)$ on $\bC\ell(\fm)$ and $\cF(\fm_J)$ in the sense that
\setcounter{equation}{4}
\begin{equation}
\label{eq5.5}
\rho_R(\k_J(c)\psi) = \k_J(\rho_R(c))(\rho_R(\psi))
\end{equation}
for all $R \in U(\fm_J)$, $c \in \bC\ell(\fm)$ and $\psi \in \cS$.
\end{proposition}

\begin{proof}
It suffices to show that
\[
\rho_R(\k_J(X)(X_1 \wedge \dots \wedge X_p)) = \k_J(\rho_R(X))(\rho_R(X_1 \wedge \dots \wedge X_p))
\]
for all $R \in U(\fm_J)$ and all $X,X_1,\dots,X_p \in \fm_J$.  Now
\begin{align*}
\rho_R(a_J^{\dag}(X)(X_1 \wedge \dots \wedge X_p)) &= \rho_R(X \wedge X_1 \wedge \dots \wedge X_p) \\
&= (R(X)) \wedge (R(X_1)) \wedge \dots \wedge (R(X_p)) \\
&= a_J^{\dag}(R(X))(\rho_R(X_1 \wedge \dots \wedge X_p)).
\end{align*}
A similar calculation, using the fact that $\rho$ preserves $\<\cdot,\cdot\>_J$, shows that
\[
\rho_R(a_J(X)(X_1 \wedge \dots \wedge X_p)) = a_J(R(X))(\rho_R(X_1 \wedge \dots \wedge X_p)).
\]
In view of how $\k_J$ is defined in terms of $a_J^{\dag}$ and $a_J$, we see that \eqref{eq5.5} holds.
\end{proof}

Since $\pi$ carries $K$ into $U(\fm_J)$, we immediately obtain:

\setcounter{theorem}{5}
\begin{corollary}
\label{cor5.6}
The actions $\pi$ of $K$ on $\bC\ell(\fm)$ and $\cF(\fm_J)$ are compatible with the action $\k_J$ of $\bC\ell(\fm)$ on $\cF(\fm_J)$ in the sense given above.
\end{corollary}

Let $\rho$ denote also the actions of the Lie algebra $u(\fm_J)$ on $\bC\ell(\fm)$ and $\cF(\fm_J)$ by derivations.  We quickly obtain the following corollary, which we will need later for our discussion of connections:

\begin{corollary}
\label{cor5.7}
The action $\k_J$ of $\bC\ell(\fm)$ on $\cF(\fm_J)$ is compatible with the actions $\rho$ of $u(\fm_J)$ on $\bC\ell(\fm)$ and $\cF(\fm_J)$ in the sense of the Leibniz rule
\setcounter{equation}{7}
\begin{equation}
\label{eq5.8}
\rho_L(\k(c)\psi) = \k(\rho_L(c))\psi + \k(c)\rho_L(\psi)
\end{equation}
for all $L \in u(\fm_J)$, $c \in \bC\ell(\fm)$ and $\psi \in \cS$.
\end{corollary}

Notice that we have never needed to use explicitly the $\spin^c$ groups in our discussion.

Next, in order to see that everything fits well, let us show that $\k_J$ respects the involutions, where by this we mean that
\setcounter{equation}{8}
\begin{equation}
\label{eq5.9}
\k_J(c^*) = (\k_J(c))^*
\end{equation}
for all $c \in \bC\ell(\fm)$, where the $^*$ on the left is the involution on $\bC\ell(\fm)$ defined earlier, while the $^*$ on the right means the adjoint of the operator for the inner product on $\cS$.  (Thus $\cS$ is a ``self-adjoint Clifford  module'' as in definition 9.3 of \cite{GVF}, but for our conventions.)  It suffices to prove this for $c = X$ for all $X \in \fm$, that is, it suffices to show that $(\k_J(X))^* = -\k_J(X)$.  In view of how $\k_J$ is defined in terms of $a_J$ and $a_J^{\dag}$ it suffices to show that
\[
(a_J(X))^* = a_J^{\dag}(X).
\]
This is well-known, and can be seen as follows.  We can assume that $\|X\|_J = 1$.  Set $e_1 = X$, and choose $e_2,\dots,e_n \in \fm_J$ such that $e_1,\dots,e_n$ is an orthonormal $\bC$-basis for $\fm_J$.  For any $I = \{j_1 < j_2 < \dots < j_p\} \subseteq \{1,\dots,n\}$ set $e_I = e_{j_1} \wedge e_{j_2} \wedge \dots \wedge e_{j_p}$ in $\cS$, and set $e_{\emptyset} = 1$.  A glance at \eqref{eq5.3} shows that $\{e_I\}$ is an orthonormal basis for $\cS$.  Then from \eqref{eq5.3} one quickly sees that
\[
\<a_J(e_1)e_{I_1},e_{I_2}\>_J = 1 \ \text{ if } \  1 \in I_1 \text{ and } I_2 = I_1\setminus \{1\},
\]
and is $0$ otherwise, while
\[
\<e_{I_1},a_J^{\dag}(e_1)e_{I_2}\>_J = 1 \ \text{ if } \ 1 \notin I_2 \text{ and } I_1 = I_2 \cup \{1\},
\]
and is $0$ otherwise.  This shows that $a_J^{\dag}(e_1)$ is the adjoint of $a_J(e_1)$.

Finally, let us consider the chirality element, following the discussion in definition 5.2 of \cite{GVF} and the paragraphs following it.  Choose an orientation for $\fm$, and let $X_1,\dots,X_{2n}$ be an oriented orthonormal $\bR$-basis for $\fm$ and $g_0$.  Define the chirality element, $\g$, of $\bC\ell(\fm)$ (for the chosen orientation) by
\[
\g = (i)^n X_1X_2 \cdots X_{2n}.
\]
(The $i$ is included because our sign convention differs from that of \cite{GVF}.)  Then, much as discussed in \cite{GVF}, $\g$ does not depend on the choice of the oriented orthonormal basis, and it satisfies $\g^2 = 1$, $\g^* = \g$ and $\g X \g = -X$ for every $X \in \fm$.  In particular, conjugation by $\g$ is the grading operator on $\bC\ell(\fm)$ that gives the even and odd parts.  Since $U(\fm_J)$ is connected, each element of $U(\fm_J)$ carries an oriented orthonormal basis into an other one, and thus leaves $\g$ invariant.  Consequently, for any $X \in u(\fm_J)$ its derivation action on $\bC\ell(\fm)$ takes $\g$ to $0$.  Since $\pi(K) \subseteq U(\fm_J)$ we have $\pi_s(\g) = \g$ for all $s \in K$, and $\pi_X(\g) = 0$ for all $X \in \fk$.  Because $\k_J$ is a $*$-representation, we will have $(\k_J(\g))^2 = 1$ and $\k_J(\g) = (\k_J(\g))^*$.  Since $\g \ne 1$, $\k_J(\g) \ne I_{\cS}$, and thus $\k_J(\g)$ will split $\cS$ into two orthogonal subspaces, $\cS^{\pm}$, the ``half-spinor'' spaces.  Because $\pi_s(\g) = \g$ for all $s$, each of $\cS^+$ and $\cS^-$ will be carried into itself by the representation $\pi$ of $K$ on $\cS$.  Because $\g X \g = -X$ for $X \in \fm$, we see that $\k_J(X)$ will carry $\cS^+$ into $\cS^-$ and $\cS^-$ into $\cS^+$ for each $X \in \fm$.  Of course each of $\cS^+$ and $\cS^-$ will be carried into itself by the subalgebra of even elements of $\bC\ell(\fm)$.


\section{Dirac operators for almost-Hermitian $G/K$}
\label{sec6}

In this section we assume as before that $G$ is a compact semisimple Lie group, but we only assume that $K$ is a closed subgroup of $G$, not necessarily connected.  We assume further only that $G/K$ is (homogeneous) almost Hermitian, by which we mean that we have an inner product, $g_0$, on $\fm$ that is invariant for the $\Ad$-action of $K$ on $\fm$, and that we have a complex structure $J$ on $\fm$ that is orthogonal for $g_0$ and commutes with the $\Ad$-action of $K$ (so $\fm$ is even-dimensional).  Then $g_0$ and $J$ are extended to $\cT(G/K)$ pointwise.  There are many examples of such coset spaces beyond the coadjoint orbits.  See for example the many constructions in sections~8 and 9 of \cite{WlG}.  But I have not seen in the literature any complete classification of all of the possibilities for almost-Hermitian compact coset spaces.  In this section we show how to construct a ``Dirac operator'' for any connection on $\cT(G/K)$ that is compatible with $g_0$ and commutes with $J$.  For coadjoint orbits we have seen that both the canonical connection and (in Theorem \ref{prop3.4}) the Levi--Civita connection $\nabla^{\dia}$ for $g_{\dia}$ commute with $J_{\dia}$.

Much as in section~7 of \cite{R22} we can form the Clifford bundle over $G/K$ for $g_0$, except that here we use the complex Clifford algebra that was discussed in the previous section instead of the real Clifford algebra used in \cite{R22}.  The role of $\pi$ of the previous section is now taken by $\Ad$ restricted to $K$ and acting on 
$\fm$, and so also on $\bC\ell(\fm) = \bC\ell(\fm,g_0)$.
From now on we will denote this action of $K$ on  $\bC\ell(\fm)$ by $\widetilde\Ad$.  We set
\begin{equation}
\label{eq6.1}
\bC\ell(G/K) = \{c \in C^{\infty}(G,\bC\ell(\fm)): c(xs) = \widetilde\Ad_s^{-1}(c(x)) \text{ for } x \in G,\ s \in K\}.
\end{equation}
It is clearly an algebra for pointwise operations.  We let $A_{\bC} = C^{\infty}(G/K,\bC)$, and we see that not only is $\bC\ell(G/K)$ an algebra over $A_{\bC}$, but that in fact $A_{\bC}$ can be identified with the center of $\bC\ell(G/K)$ (since $\fm$ is even-dimensional).  Furthermore, $\bC\ell(G/K)$ contains the tangent bundle $\cT(G/K)$ of $G/K$ as a real (generating) subspace. On $\bC\ell(G/K)$ we have the action $\l$ of $G$ by translation, and this action defines the canonical connection, $\nabla^c$, on $\bC\ell(G/K)$, which acts by derivations (much as discussed in \cite{R22}). Clearly this $\nabla^c$ extends the $\nabla^c$ on $\cT(G/K)$.

Suppose now that $\nabla$ is some other connection on $\cT(G/K)$ that is $G$-invariant and compatible with $g_0$ (such as our earlier $\nabla^{\dia}$ when $G/K$ is a coadjoint orbit).  As seen in section~5 of \cite{R22}, especially corollary~5.2, $\nabla$ is then of the form $\nabla = \nabla^c + L$ where $L$ is a $G$-equivariant $A$-homomorphism from $\cT(G/K)$ into $\End_A^{sk}(\cT(G/K))$.  Here $\End_A^{sk}(\cT(G/K))$ denotes the $A$-endomorphisms of $\cT(G/K)$ that are skew-adjoint with respect to $g_0$.  As seen in proposition 3.1 of \cite{R22}, each such endomorphism $L_V$ for $V \in \cT(G/K)$ is given by a smooth function on $G$ whose values are in $so(\fm,g_0)$, which we denote again by $L_V$, and which satisfies the condition
\[
L_V(xs) = \Ad_s^{-1} \circ L_V(x) \circ \Ad_s
\]
for $x \in G$ and $s \in K$.  For any $V,W \in \cT(G/K)$ we have $(L_VW)(x) = (L_V(x)(W(x))$ for $x \in G$.  By equation \eqref{eq5.2} each $L_V$ will extend to a derivation of $\bC\ell(G/K)$, and in this way we obtain an $A$-linear (so $\bR$-linear) map from $\cT(G/K)$ into the Lie-algebra of derivations of $\bC\ell(G/K)$.  (These derivations will, in fact, be $*$-derivations for the involution determined by the involution on $\bC\ell(\fm)$ defined in Section~\ref{sec5}.)  We can now define a $G$-invariant connection, $\nabla$, on $\bC\ell(G/K)$ by $\nabla = \nabla^c + L$.  It clearly extends the original $\nabla$ on $\cT(G/K)$ (and $\d$ on $A_{\bC}$).  Again $\nabla_V$ will be a derivation of $\bC\ell(G/K)$ for each $V \in \cT(G/K)$.
Note that our construction of $\bC\ell(G/K)$ and its $\nabla$  does not use $J$.

We use $J$ in the way described in the previous section to define the complex Hilbert space $\cS = \cF(\fm_J)$ of spinors, with its compatible actions of $\bC\ell(\fm)$ and $K$.  We will again denote the action of $K$ on $\cS$ by $\widetilde\Ad$. We then define the canonical bundle $\cS(G/K)$ of spinor fields on $G/K$ for $J$ by
\begin{equation}
\label{eq6.2}
\cS(G/K) = \{\psi \in C^{\infty}(G,\cS): \psi(xs) = 
\widetilde\Ad_s^{-1}(\psi(x)) \text{ for } x \in G,\ s \in K\}.
\end{equation}
It is an $A_{\bC}$-module in the evident way (projective by proposition 2.2 of \cite{R22}).  

As explained in theorem 1.7i of \cite{Ply} and proposition 9.4 of \cite{GVF} and later pages, spinor bundles for Clifford bundles are not in general unique. The tensor product of a spinor bundle by a line bundle will be another spinor bundle, and all the spinor bundles are related in this way. Within our setting of equivariant bundles we need to tensor with $G$-equivariant line bundles. These correspond exactly to the characters, that is, one-dimensional representations, of $K$. (We will not discuss Dirac operators twisted by vector bundles of higher dimension.) From theorem 5.1 of \cite{R22} it is easily seen that $G$-invariant connections on a line bundle differ from the canonical connection by a constant. For our purposes we can ignore the constant. In fact, even the canonical connection need not appear explicitly. We proceed as follows. Let $\chi$ be a character of $K$. (We remark that when $G/K$ is a coadjoint orbit, $K$ always has nontrivial characters because $\ft$ is an ideal in $\fk$.) We set:
\begin{equation*}
\cS(G/K, \chi) = \{\psi \in C^{\infty}(G,\cS): \psi(xs) = 
\bar\chi(s)\widetilde\Ad_s^{-1}(\psi(x)) \text{ for } x \in G,\ s \in K\}.
\end{equation*}
On $\cS(G/K, \chi)$ we define an $A_{\bC}$-valued inner product in the usual way by
\[
\<\psi,\var\>_{A_{\bC}}(x) = \<\psi(x),\var(x)\>_{\cS}
\]
for $\psi,\var \in \cS(G/K, \chi)$.  Of greatest importance is the action $\k$ of $\bC\ell(G/K)$ on $\cS(G/K, \chi)$ that is defined by
\begin{equation}
\label{eq6.3}
(\k(c)\psi)(x) = \k(c(x))(\psi(x))
\end{equation}
for $x \in G$.  (We drop the subscript $J$ on $\k$ used in the previous section.)  This action carries $\cS(G/K, \chi)$ into itself because
\[
(\k(c)\psi)(xs) = (\k(\widetilde\Ad_s^{-1}(c(x))))
(\bar\chi(s)\widetilde\Ad_s^{-1}(\psi(x))) = \bar\chi(s)\widetilde\Ad_s^{-1}((\k(c)\psi)(x)),
\]
where the last equality follows from Proposition~\ref{prop5.4}.

On $\cS(G/K, \chi)$ we have the action $\l$ of $G$ by translation, and it is easily seen that $\k$ is compatible with this action and the $G$-action on $\bC\ell(G/K)$.  The action $\l$ defines a canonical connection on $\cS(G/K, \chi)$ by adapting \eqref{eq2.1} in the evident way.  We will denote this canonical connection again by $\nabla^c$.  Of prime importance, we have the Leibniz rule
\begin{equation}
\label{eq6.4}
\nabla_V^c(\k(c)\psi) = \k(\nabla_V^cc)\psi + \k(c)(\nabla_V^c\psi)
\end{equation}
for any $V \in \cT(G/K)$, $c \in \bC\ell(G/K)$ and $\psi \in \cS(G/K, \chi)$.  Furthermore, much as discussed in \cite{R22}, the connection on $\cS(G/K, \chi)$ is compatible with the $A_{\bC}$-valued inner product in the sense of the Leibniz rule
\begin{equation}
\label{eq6.5}
\d_V(\<\psi,\var\>_{A_{\bC}}) = \<\nabla_V^c\psi,\var\>_{A_{\bC}} + \<\psi,\nabla_V^c\var\>_{A_{\bC}}
\end{equation}
for any $V \in \cT(G/K)$ and $\psi,\var \in \cS(G/K, \chi)$.  The only property of $J$ that is used for this is the evident fact that when we view $J$ as acting on $\cT(G/K)$ pointwise, it commutes with the translation action of $G$.

Suppose now that our original $\nabla$ on $\cT(G/K)$ commutes with $J$, in the sense that each $\nabla_V$ does.   As before, set $\nabla = \nabla^c + L$. Since also $\nabla^c$ commutes with $J$, each $L_V$ will commute with $J$, that is, $L_V(x) \in u(\fm_J)$ for each $x \in G$.  Then, as discussed in the previous section, each $L_V(x)$ will extend to a derivation of the exterior algebra $\cS = \cF(\fm_J)$, and consequently $L_V$ determines an $A$-module endomorphism of $\cS(G/K, \chi)$, which we denote by $L_V^{\cS}$.  In this way we define a $G$-equivariant $A$-linear map $L^{\cS}$ from $\cT(G/K)$ into the algebra of $A_\bC$-endomorphisms of $\cS(G/K, \chi)$.  Furthermore, it is easily checked that each $L_V^{\cS}$ is skew-adjoint for the $A_{\bC}$-valued inner product on $\cS(G/K, \chi)$.  Of most importance, we see from Corollary~\ref{cor5.7} that $L^{\cS}$ is compatible with the action 
of $\bC\ell(G/K)$ on $\cS(G/K, \chi)$ in the sense of the Leibniz rule
\[
L_V^{\cS}(\k(c)\psi) = \k(L_Vc)\psi + \k(c)(L_V^{\cS}\psi)
\]
for all $V$, $c$ and $\psi$.  We saw in \eqref{eq6.4} that $\nabla^c$ satisfies a similar identity, and so we have obtained:

\setcounter{theorem}{5}
\begin{proposition}
\label{prop6.6}
Let $(G/K,g_0,J)$ be almost Hermitian, and let $\chi$ be a character of $K$. Let $\nabla$ be a $G$-invariant connection on $\cT(G/H)$ that is compatible with $g_0$ and commutes with $J$.  Let $\nabla$ also denote its extension to a connection on $\bC\ell(G/K)$ as constructed above, and let $\nabla^{\cS}$ denote the corresponding connection on $\cS(G/K, \chi)$ constructed above using $J$.  Then the Leibniz rule
\[
\nabla_V^{\cS}(\k(c)\psi) = \k(\nabla_Vc)\psi + \k(c)(\nabla_V^{\cS}\psi)
\]
holds for all $V \in \cT(G/K)$, $c \in \bC\ell(G/K)$ and $\psi \in \cS(G/K, \chi)$.  Furthermore, $\nabla^{\cS}$ is compatible with the $A_{\bC}$-valued inner product on $\cS(G/K, \chi)$ from $g_0$.
\end{proposition}

In terms of the connection $\nabla^{\cS}$ on $\cS(G/K, \chi)$ from $\nabla$ we can define the ``Dirac'' operator for $\nabla$ (which when $G/K$ is a coadjoint orbit will be the canonical Dirac operator for $\mu_\dia$ when $\chi$ is trivial, $\cS(G/K)$ is constructed using $J_\dia$, and $\nabla^\dia$ is the Levi--Civita connection for $g_\dia$).  Much as done in section~8 of \cite{R22}, for any $\psi \in \cS(G/K, \chi)$ we define $d\psi$ by $d\psi(V) = \nabla_V^{\cS}(\psi)$ for $V \in \cT(G/K)$.  Then we can view $d\psi$ as an element of $\cT^*(G/K) \otimes_{\bR} \cS(G/K, \chi)$, where $\cT^*(G/K)$ denotes the $A$-module of smooth cross-sections of the cotangent bundle.  By means of the Riemannian metric $g_0$ (as $A$-valued inner product) we can identify $\cT^*(G/K)$ with $\cT(G/K)$.  When $d\psi$ is viewed by this identification as an element of $\cT(G/K) \otimes_{\bR} \cS(G/K, \chi)$ we will denote it, with some abuse of notation, by $\grad^0\psi$.  Let us view the Clifford action $\k$ of $\bC\ell(G/K)$ on $\cS(G/K, \chi)$ as a bilinear mapping from $\bC\ell(G/K)\otimes_{\bC} \cS(G/K, \chi)$ into $\cS(G/K, \chi)$.  We can view $\cT(G/K)$ as a real subspace of $\bC\ell(G/K)$ in the evident way, and so we can view $\cT(G/K) \otimes_{\bR} \cS(G/K, \chi)$ as a real subspace of $\bC\ell(G/K) \otimes_\bC \cS(G/K, \chi)$.  In this way we view $\grad_{\psi}^0$ as an element of $\bC\ell(G/K) \otimes_{\bC} \cS(G/K, \chi)$, to which we can apply $\k$.

\begin{definition}
\label{def6.7}
Let $(G/K,g_0,J)$ be almost Hermitian, and let $\chi$ be a character of $K$. Let $\nabla$ be a $G$-invariant  connection on $\cT(G/K)$ that is compatible with $g_0$ and commutes with $J$.  Then the Dirac operator, $D^{\nabla}$, for $\nabla$ and $\chi$ is defined on $\cS(G/K, \chi)$ by
\[
D^{\nabla}\psi = \k(\grad_{\psi}^0).
\]
\end{definition}

We remind the reader that $\k$ depends on the choice of $g_0$ and $J$, and that $\grad^0\psi$ also depends on the choice of $\nabla$.

In the setting of Definition~\ref{def6.7} we can use a standard module frame $\{W_j\}$ for $\cT(G/K)$ and $g_0$ to give a more explicit description of $\grad_{\psi}^0$, namely
\[
\grad_{\psi}^0 = \sum_j W_j \otimes (\nabla_{W_j}^{\cS}\psi).
\]
(See the paragraph of \cite{R22} containing equation 8.2.)  In terms of $\{W_j\}$ we can then write $D^{\nabla}$ as
\setcounter{equation}{7}
\begin{equation}
\label{eq6.8}
D^{\nabla}\psi = \sum_j \k(W_j)(\nabla_{W_j}^{\cS}\psi).
\end{equation}
These expressions for $\grad_{\psi}^0$ and $D^{\nabla}$ are, of course, independent of the choice of standard module frame.  This can be seen directly by using Proposition~\ref{prop2.9}.

\setcounter{theorem}{8}
\begin{proposition}
\label{prop6.9}
The operator $D^{\nabla}$ commutes with the action of $G$ on $\cS(G/K, \chi)$ by translation, and anti-commutes with the chirality operator $\k(\g)$.
\end{proposition}

\begin{proof}
The commutation with the action of $G$ is easily verified, much as done in the paragraph after equation 8.2 of \cite{R22}.  As to $\k(\g)$, we are viewing $\g$ as a constant field in $\bC\ell(G/K)$, and because $L_X \in u(\fm_J)$ for each $X \in \fm$ we have $L_V\g = 0$, as seen near the end of the previous section.  Since $\g$ is constant, we also clearly have $\nabla_V^c\g = 0$, and thus $\nabla_V\g = 0$ for all $V \in \cT(G/K)$.  Then from Proposition~\ref{prop6.6} we see that $\k(\g)$ commutes with $\nabla_V^{\cS}$ for each $V \in \cT(G/K)$.  Since $\g$ anti-commutes with each $X \in \fm \subset \bC\ell(\fm)$, it follows easily that $D^{\nabla}$ anti-commutes with $\k(\g)$.
\end{proof}

We have been viewing $\cS(G/K, \chi)$ as a right $A_{\bC}$-module.  But since $A_{\bC}$ is commutative we can equally well view $\cS(G/K, \chi)$ as a left $A_{\bC}$-module, and this view is quite natural when we view $A_{\bC}$ as the center of $\bC\ell(G/K)$.  For any $f \in A_{\bC}$ we let $M_f$ denote the operator on $\cS(G/K, \chi)$ consisting of pointwise multiplication by $f$, viewed as acting on the left of $\cS(G/K, \chi)$.  By means of $g_0$ we can identify $df$ (defined by $df(V) = \d_V(f)$) with an element of $\cT_{\bC}(G/K)$, which we denote by $\grad_f^0$ since it is the usual gradient of $f$ for $g_0$.  In terms of a standard module frame, $\{W_j\}$, for $\cT(G/K)$ we have
\[
\grad_f^0 = \sum_j (\d_{W_j}f)W_j,
\]
with the evident meaning considering that $\d_{W_j}f$ is $\bC$-valued.  Then it is easily seen, much as in the proof of proposition~8.3 of \cite{R22}, that:

\begin{proposition}
\label{prop6.10}
For any $f \in A_{\bC}$ and $\psi \in \cS(G/K, \chi)$ we have
\[
[D^{\nabla},M_f]\psi = \k(\grad_f^0)(\psi).
\]
\end{proposition}

For the reader's convenience we now basically repeat the comments made right after the proof of theorem~8.4 of \cite{R22}. Let the Hilbert space $L^2(G/K, \cS)$ be defined in terms of the $G$-invariant measure on $G/K$ from that on $G$. By choosing a fundamental domain in $G$ we can view $\cS(G/K, \chi)$ as a dense subspace of $L^2(G/K, \cS)$. In this way $D^\nabla$ can be viewed as an unbounded operator on $L^2(G/K, \cS)$. Note that for different choices of $\chi$ the spectrum of $D^\nabla$ can be quite different.  We equip $\cS(G/K, \chi)$ with the inner product from $L^2(G/K, \cS)$, which will just be $\int_G \<\psi,\var\>_{A_{\bC}}$. For $f \in A_\bC$ we let $M_f$ denote also the corresponding operator on $L^2(G/K, \cS)$ by pointwise multiplication.  From Proposition~\ref{prop6.10} we see that the operator norm of the commutator $[D^{\nabla},M_f]$ is the same as that of $\k(\grad_f^0)$ as an operator on $\cS(G/K)$. Recall from equation \ref{eq5.9} that $\k$ is a $*$-representation.
For any $c \in \bC\ell(G/K)$ let $\|\k(c)\|$ denote the operator norm of $\k(c)$ as an operator on $\cS(G/K, \chi)$.  Then by the $C^*$-identity $\|T\|^2 =\|T^*T\|$ we see that $\|\k(c)\|^2 = \|\k(c^*c)\|$.  When $c = V \in \cT(G/K)$ this means that
\[
\|\k(V)\|^2 = \|\k(\<V,V\>_A)\| = \|M_{\<V,V\>_A}\| = \|\<V,V\>_A\|_{\infty} = \|V\|_{\infty}^2,
\]
for the evident meaning of the last term, where $\|\cdot\|_{\infty}$ is just the usual supremum norm. Notice that this is independent of the choice of $\chi$ (basically reflecting the fact that the $C^*$-norm on a full matrix algebra is unique). When we apply this for $V = \grad_f^0$ we obtain
\[
\|\k(\grad_f^0)\| = \|\grad_f^0\|_{\infty}.
\]
Now a standard argument (e.g., following definition~9.13 of \cite{GVF}) shows that if we denote by $\rho$ the ordinary metric on a Riemannian manifold $N$ coming from its Riemannian metric, then for any two points $p$ and $q$ of $N$ we have
\[
\rho(p,q) = \sup\{|f(p)-f(q)|: \|\grad_f\|_\infty \le 1\}.
\]
On applying this to $G/K$ and using Proposition~\ref{prop6.10} and the discussion following its proof, we obtain, for $\rho$ now the ordinary metric on $G/K$ from our Riemannian metric $g_0$,
\[
\rho(p,q) = \sup\{|f(p)-f(q)|: \|[D^{\nabla},M_f]\| \le 1\}.
\]
This is the formula on which Connes focused for general Riemannian manifolds \cite{Cn7,Cn3}, as it shows that the Dirac operator contains all of the metric information (and much more) for the manifold.  This is his motivation for advocating that metric data for ``non-commutative spaces'' be encoded by providing them with a ``Dirac operator''.  But we should notice that our Dirac operators above may not be formally self-adjoint.  We deal with that issue in the next section.

We remark that the first part of Proposition~\ref{prop6.9} is the manifestation in terms of $D^{\nabla}$ of the fact that the ordinary metric on $G/K$ for $g_0$ is invariant for the action of $G$ on $G/K$.

At this point it is clear that we can combine the construction of this section with the formula in Theorem~\ref{thm3.3} for the Levi--Civita connection for a coadjoint orbit to obtain a fairly explicit formula for the canonical Dirac operator for the coadjoint orbit of $\mu \in \fg'$. But we refrain from writing this formula here as it is somewhat lengthy, and we do not need it for the next section.


\section{The formal self-adjointness of the Dirac operator}
\label{sec7}

By definition, $D^\nabla$ will be formally self-adjoint if
\[
\<D^{\nabla}\psi,\var\> = \<\psi,D^{\nabla}\var\>
\]
for any $\var,\psi \in \cS(G/K, \chi)$, where the inner product is that from $L^2(G/K,\cS)$.
Recall that the torsion, $T_{\nabla}$, of a connection $\nabla$ on $\cT(G/K)$ is defined by
\[
T_{\nabla}(V,W) = \nabla_V W - \nabla_W V - [V,W]
\]
for $V,W \in \cT(G/K)$. Note that $[V,W]$ is defined as the commutator of derivations of $A$, and that when elements of $\cT(G/K)$ are viewed as functions as we have been doing, then $[V,W]$ is not defined pointwise, but rather has a somewhat complicated expression in terms of $V$ and $W$. But in section 6 of \cite{R22} it is seen that the function $[V,W]$ can be readily calculated when $V$ and $W$ are fundamental vector fields, and we will use this fact later. It is not difficult to see that $T_\nabla$ is $A$-bilinear. (See \S 8 of chapter 1 of \cite{Hlg}.) 
For any $U \in \cT(G/K)$ let $T_\nabla^U$ be the $A$-endomorphism of
$\cT(G/K)$ defined by $T_\nabla^U(V) = T_\nabla(U, \ V)$. We can define 
$\mathrm{trace}(T^U_\nabla)$ by
$\mathrm{trace}(T^U_\nabla) = \sum_j g_0(T_\nabla(U, W_j), W_j )$
for one (hence every, by Proposition \ref{prop2.9})
standard module frame $\{W_j\}$ for $\cT(G/K)$ equipped with $g_0$.

The purpose of this section is to prove the following theorem, and to obtain some of its consequences.
As we will see, this theorem is closely related to the main theorem of \cite{Ikd}, which deals with the case in which $G/K$ is spin. (See also \cite{FrS}.)

\begin{theorem}
\label{th7.1}
Let $(G/K,g_0, J)$ be almost Hermitian, and let $\chi$ be a character of $K$. Let $\nabla$ be a $G$-invariant connection on $\cT(G/K)$ that is compatible with $g_0$ and commutes with $J$, so that we can define the Dirac operator $D^{\nabla}$ on $L^2(G/K,\cS)$, with domain $\cS(G/K, \chi)$, as explained in the previous section.  Then $D^{\nabla}$ is formally self-adjoint if and only if
\[
\mathrm{trace}(T_\nabla^U)= 0
\]
for every $U \in \cT(G/K)$.
\end{theorem}

\begin{proof}
We try to follow the path of the proof of theorem~8.4 in the latest revised arXiv version of \cite{R22}.  (The published version of this paper has a serious error in the proof of theorem~8.4, and that error is corrected in the most recent arXiv version.)  We use first the Leibniz rule of Proposition~\ref{prop6.6} and then the compatibility of $\nabla^{\cS}$ with the inner product to calculate that for $\psi,\var \in \cS(G/K, \chi)$ we have
\setcounter{equation}{1}
\begin{equation}
\label{eq7.2}
\begin{aligned}
&\<D^{\nabla}\psi, \ \var\>_{A_{\bC}} - \<\psi, \ D^{\nabla}\var\>_{A_{\bC}} \\
&= \sum_j (\<\k(W_j)(\nabla_{W_j}^{\cS}\psi), \ \var\>_{A_{\bC}} - \<\psi, \ \k(W_j)(\nabla_{W_j}^{\cS}\var)\>_{A_{\bC}}) \\
&= \sum_j (-\<\nabla_{W_j}^{\cS}\psi, \ \k(W_j)\var\>_{A_{\bC}}) - \<\psi, \ \nabla_{W_j}^{\cS}(\k(W_j)\var) - \k(\nabla_{W_j}W_j)\var\>_{A_{\bC}}) \\
&= -\sum_j \d_{W_j}(\<\psi, \ \k(W_j)\var\>_{A_{\bC}}) + \<\psi, \ \k(\sum_j\nabla_{W_j}W_j)\var\>_{A_{\bC}}.
\end{aligned}
\end{equation}
For given $\psi$ and $\var$ the function $V \mapsto \<\psi,\k(V)\var\>_{A_{\bC}}$ for $V \in \cT(G/K)$ is $A$-linear. It is \emph{not} $\bC$-linear for the complex structure on $\cT(G/K)$ from $J$, because $\k$ is not $\bC$-linear, as was mentioned immediately after the definition of $\k = \k_J$ in Section \ref{sec5}.
Of course, the above function does extend to an $A_\bC$-linear function from the complexification, $\cT_\bC(G/K)$, of $\cT(G/K)$. We equip $\cT_\bC(G/K)$ with the complexification of the inner product on $\cT(G/K)$ from $g_0$. Clearly $\cT_\bC(G/K)$ corresponds to the ``induced bundle'' for the $\Ad$-action of $K$ on the complexification, $\fm_\bC$, of $\fm$. Every $A_\bC$-linear function from $\cT_\bC(G/K)$ into $A_\bC$ is represented through the inner product by an element of $\cT_\bC(G/K)$. (See, e.g., proposition 7.2 of \cite{R17}.) Thus there is a $U \in \cT_\bC(G/K)$ such that 
 $\<\psi,\k(V)\var\>_{A_{\bC}} = \<U,V\>_{A_{\bC}}$ for all $V \in \cT(G/K)$. 
 
 \setcounter{theorem}{2}
\begin{lemma}
\label{lem7.3a}
Let $\mathcal{E}$ denote the $\bC$-linear span of the $U$'s that arise as above from pairs $(\psi, \phi)$ of elements of $ \cS(G/K, \chi)$. Then $\mathcal E = \cT_\bC(G/H)$.
\end{lemma}
\begin{proof}
It suffices to show that $\cT(G/K)$ is in $\mathcal E$. So let $U \in  \cT(G/K)$ be given. Let also a cross-section $h$ for the line-bundle for $\chi$ be given, so that $h \in C^\infty(G, \bC)$ and $h(xs) = \bar \chi(s)h(x)$ for all $x \in G$ and $s \in K$. Let $1_\cS$ denote the identity element of $\cS = \cF(m_J)$, and let $\phi$ be defined by $\phi(x) = h(x) 1_\cS$. View $U$ as having values in $\fm_J$, and let $\psi$ be defined by $\psi(x) = h(x)U(x)$, using $J$ to define the $\bC$-space structure of $\fm_J$. Then both $\psi$ and $\phi$ are in $\cS(G/K, \chi)$. For any $V \in \cT(G/K)$ we then have
\[
(\k(V)\phi)(x) = ia_J^\dag(V(x))(h(x)1_\cS) = ih(x)V(x).
\]
Thus, with $\<\cdot, \cdot\>_J$ defined on $\fm_J$ as done in Section \ref{sec5} and on $\cS$ as done in equation \ref{eq5.3}, we have
\begin{eqnarray*}
\<\psi, \k(V)\phi\>_{A_\bC}(x) &=& i|h(x)|^2\<U(x), V(x)\>_J \\
&=& i|h(x)|^2(g_0(U(x), V(x)) + ig_0(JU(x), V(x))) .
\end{eqnarray*}
But because $\k(V)$ is a skew-adjoint operator on $\cS(G/K,\chi)$, we have
\[
\<\psi, \k(V)\phi\>_{A_\bC}^- = \<\k(V)\phi, \psi\>_{A_\bC} = - \<\phi, \k(V)\psi\>_{A_\bC} .
\]
Thus the real and imaginary parts of the function $V \mapsto \<\psi, \k(V)\phi\>_{A_\bC}$ are both in $\mathcal E$, and so that the function $V \mapsto |h|^2g_0(U,V)$ is in $\mathcal E$. Note that $|h|^2 \in A$. Now let $\{h_j\}$ be a standard module frame for the line-bundle for $\chi$, so that $\sum |h_j|^2 = 1$. Since each function $V \mapsto |h_j|^2g_0(U,V)$ is in $\mathcal E$, by summing them over $j$ we see that the function $V \mapsto g_0(U,V)$ is in $\mathcal E$, as desired.
\end{proof}
 
Now in terms of the $U$ for $\psi$ and $\var$ the expression \eqref{eq7.2} becomes:
\[
= -\sum_j \d_{W_j}(\<U,W_j\>_{A_{\bC}}) + \<U, \sum_j\nabla_{W_j}W_j\>_{A_{\bC}}) = \sum_j \<\nabla_{W_j}U,W_j\>_{A_{\bC}}.
\]
Thus from Lemma \ref{lem7.3a} we see that $D^{\nabla}$ is formally self-adjoint if and only if
\[
\int_{G/K} \sum_j \<\nabla_{W_j}U,W_j\>_{A_{\bC}} = 0
\]
for all $U \in \cT_{\bC}(G/K)$.  By expressing the real and imaginary parts of the inner product in terms of $g_0$, and expressing $U$ in terms of its real and imaginary parts, we see that we have obtained:

\begin{lemma}
\label{lem7.3}
With notation as above, the Dirac operator $D^{\nabla}$ is formally self-adjoint if and only if
\[
\int_{G/K} \sum_j g_0(\nabla_{W_j}U,W_j) = 0
\]
for all $U \in \cT(G/K)$ and one (hence every) standard module frame, $\{W_j\}$, for $\cT(G/K)$ and $g_0$.
\end{lemma}

Thus we have reduced the matter to a condition concerning $\nabla$ on $\cT(G/K)$, so $J$ is no
longer involved, we no longer need to consider the Clifford algebra and spinors, and we can work over $\bR$ from this point on.
From the definition of the torsion, $T_{\nabla}$, of $\nabla$ we have
\[
g_0(\nabla_{W_j}U, \ W_j) = g_0(\nabla_UW_j - T_{\nabla}(U,W_j) - [U,W_j], \ W_j)
\]
for each $j$.  Notice now that $\sum_j g_0(W_j,W_j)$  is independent of the choice of standard module frame, by the $A$-bilinearity of $g_0$ and by Proposition~\ref{prop2.9}.  For any given $x \in G$ we have
\[
\sum_j(g_0(W_j,W_j))(x) = \sum_jg_0(W_j(x),W_j(x)),
\]
and $\{W_j(x)\}$ forms a frame for $\fm$ with $g_0$.  By Proposition~\ref{prop2.9} the expression on the right is independent of the choice of frame for $\fm$, and so we can use an orthonormal basis for $\fm$ and $g_0$.  From this we see that
\[
\sum_j g_0(W_j,W_j) \equiv \dim(\fm).
\]
Consequently, by the compatibility of $\nabla$ with $g_0$, for any $U \in \cT(G/K)$ we have
\begin{align*}
0 = \d_U(\sum_j g_0(W_j,W_j)) &= \sum_j g_0(\nabla_UW_j,W_j) + g_0(W_j,\nabla_UW_j) \\
&= 2 \sum_j g_0(\nabla_UW_j,W_j).
\end{align*}
Thus
\[
\sum_j g_0(\nabla_UW_j,W_j) = 0.
\]
(We remark that this fact depends on the pointwise argument used just above, and that the analogous argument can fail for modules over a non-commutative $A$ that contains proper isometries.)
We see thus that
\[
\sum_j g_0(\nabla_{W_j}U,W_j) = -\sum_j g_0(T_{\nabla}(U,W_j),W_j) - \sum_j g_0([U,W_j],W_j).
\]

Let $\nabla^t$ be the Levi--Civita connection for $g_0$.  We can apply the above equation to $\nabla^t$ and use the fact that $\nabla^t$ is torsion-free to get an expression for the last term above.  In this way we find that
\[
\sum_j g_0(\nabla_{W_j}U,W_j) = -\sum_j g_0(T_{\nabla}(U,W_j),W_j) + \sum g_0(\nabla_{W_j}^tU,W_j).
\]
Because $\nabla^t$ is the Levi--Civita connection for $g_0$, the sum $\sum_j g_0(\nabla^t_{W_j}U,W_j)$ is exactly $\div(U)$ as defined in Definition~\ref{def2.7}.  From the divergence theorem that was proved in Theorem~\ref{th2.10} we have
\[
\int_{G/K} \sum_j g_0(\nabla_{W_j}^tU,W_j) = 0
\]
for all $U \in \cT(G/K)$.  Thus we see that $D^\nabla$ is formally self-adjoint exactly
if
\[
\int_{G/K} \sum_j g_0(T_{\nabla}(U,W_j),W_j) = 0
\]
for all  $U \in \cT(G/K)$. But the integrand is clearly $A$-linear in $U$, so if we replace
$U$ by $Uf$ for any $f \in A$ the $f$ comes outside the inner product and the sum. Since
$f$ is arbitrary, this means that the integral is always 0 exactly if 
\[
\sum_j g_0(T_{\nabla}(U,W_j),W_j) = 0
\]
for all  $U \in \cT(G/K)$. But the left-hand side is exactly our definition of $\mathrm{trace}(T_\nabla^U)$.
 \end{proof}

Note that the criterion in the theorem is independent of the choice of $J$ (as long as $J$ commutes with $\nabla$).  

\begin{corollary}
\label{cor7.4}
Let $\mu_{\dia} \in \fg'$ and let $G/K$ correspond to the coadjoint orbit of $\mu_{\dia}$.  Let $\fg_{\dia}$ be the Riemannian metric on $G/K$ corresponding to the K\"ahler structure from $\mu_{\dia}$, and let $\nabla^{\dia}$ be its Levi--Civita connection.  Let $D^{\dia}$ be the Dirac operator for  $\nabla^{\dia}$ constructed as in the previous section (since $\nabla^{\dia}$ commutes with $J_{\dia}$), for any character of $K$.   Then $D^{\dia}$ is formally self-adjoint.
\end{corollary}

\begin{proof}
Since the torsion of $\nabla^{\dia}$ is $0$ by definition, application of Theorem~\ref{th7.1} immediately shows that $D^{\dia}$ is formally self-adjoint.
\end{proof}

For any almost-Hermitian $(G/K,g_0,J)$ there is always at least one connection that satisfies the hypotheses of Theorem~\ref{th7.1}, namely the canonical connection $\nabla^c$.  Even though it may not be torsion-free, we have:

\begin{corollary}
\label{cor7.5}
Let $(G/K,g_0,J)$ be almost Hermitian, and let $\nabla^c$ be the canonical connection on $\cT(G/K)$ . Let $D^{\nabla^c}$ be the Dirac operator constructed as in the previous section for $g_0$ using $J$ (since $\nabla^c$ is compatible with $g_0$ and commutes with $J$), for any character of $K$.  Then $D^{\nabla^c}$ is formally self-adjoint.
\end{corollary}

\begin{proof}
From section~6 of \cite{R22} (where the canonical connection is denoted by $\nabla^0$) we find that
\[
(T_{\nabla^c}(V,W))(x) = -P[V(x),W(x)].
\]
Thus to apply the criterion of Theorem~\ref{th7.1} we need to show that
\setcounter{equation}{5}
\begin{equation}
\label{eq7.6}
\mathrm{trace}(T^U_{\nabla^c})(x) = \sum_j g_0(P[U(x),W_j(x)],W_j(x)) = 0
\end{equation}
for each $U \in \cT(G/K)$ and $x \in G$.  Now $\{W_j(x)\}$ is a frame for $\fm$ with respect to $g_0$ for each $x$, and by Proposition~\ref{prop2.9} for a given $x$ we can replace $\{W_j(x)\}$ by an orthonormal basis for $\fm$ with respect to $g_0$.  We see in this way that for a given $x$, if we set $Y = U(x)$, then expression \eqref{eq7.6} is simply $\mathrm{trace}(P \circ \ad_Y)$ where $P \circ \ad_Y$ is viewed as an operator on $\fm$.  
But the trace of an operator is independent of any choice of inner product on the vector space. Thus we can instead use a basis, $\{X_j\}$, for $\fm$ that is orthonormal for $\Kil$. Since $P$ is self-adjoint for $\Kil$ on $\fg$, 
the expression \eqref{eq7.6} (for the given $x$) is just
\[
\sum \Kil([Y,X_j],X_j).
\]
But $\ad_Y$ is skew-adjoint for $\Kil$ on $\fg$, and so each term in the sum is 0.
Thus the criterion of Theorem~\ref{th7.1} is fulfilled.
\end{proof}

Suppose now that $(G/K,g_0,J)$ is almost Hermitian and that $\nabla$ is a $G$-invariant connection on $ \cT(G/K)$ that is compatible with $g_0$ and commutes with $J$. As done earlier, set $L = \nabla - \nabla^c$. Then a simple calculation shows that
\[
T_\nabla(V, W) = T_{\nabla^c}(V, W) + L_V W - L_W V,
\]
for $V, W \in  \cT(G/K)$, and so for any $U \in  \cT(G/K)$ we have
\[
\mathrm{trace}(T^U_\nabla) 
= \mathrm{trace}(T^U_{\nabla^c}) + \sum_j g_0(L_U W_j - L_{W_j} U , \ W_j).
\]
But in the proof of corollary \ref{cor7.5} we verified equation \ref{eq7.6}, which says that 
$ \mathrm{trace}(T^U_{\nabla^c}) = 0$. Furthermore, $L_U$ is skew-symmetric, so $g_0(L_U W_j, W_j) = 0$ for each $j$. It follows that
\[
\mathrm{trace}(T^U_\nabla) = -\sum_j g_0(L_{W_j} U, \ W_j) = \sum_j g_0(U, \ L_{W_j} W_j).
\]
Since we need this to be 0 for all $U$, we obtain:

\begin{corollary}
\label{cor7.7}
Let $(G/K,g_0, J)$ be almost Hermitian, and let $\nabla$ be a $G$-invariant connection on $\cT(G/K)$ that is compatible with $g_0$ and commutes with $J$. Let $D^{\nabla}$ be the Dirac operator for $g_0$ and $J$, for a character $\chi$ of $K$.  Let $L = \nabla - \nabla^c$. Then $D^{\nabla}$ is formally self-adjoint if and only if
\[
\sum_j L_{W_j} W_j = 0
\]
for one, hence every, standard module frame for $\cT(G/K)$ and $g_0$.
\end{corollary} 

The next results are motivated by the corollary in \cite{Ikd}.

\begin{lemma}
\label{lem7.8}
For $L$ as above, $\sum_j L_{W_j} W_j$ is a constant function on $G$, whose value is in the subspace of $\fm$ consisting of elements that are invariant under the $Ad$-action of $K$ on $\fm$.
\end{lemma}

\begin{proof} Because $\nabla$ and $\nabla^c$ are $G$-invariant, so is $L$, where this means that $\l_x(L_W V) = L_{\l_x W}(\l_x V)$, as seen in section 5 of \cite{R22}. Consequently for any $x \in G$
\[
(\sum_j L_{W_j} W_j)(x^{-1}) = (\sum L_{\l_x W_j} (\l_x W_j))(e) ,
\]
where $e$ is the identity element of $G$. But $\{\l_x W_j\}$ is again a standard module frame, and the espression is independent of the choice of standard module frame by Proposition \ref{prop2.9}, so the first statement is verified. For any $x \in G$ and $s \in K$ we have
\begin{eqnarray*}
\Ad_s((\sum L_{W_j})(x)) &=& \Ad_s(\sum L_{W_j}(x)(W_j(x))  \\
&=& \sum (\Ad_s \circ L_{W_j}(x) \circ \Ad_s^{-1})(\Ad_s(W_j(x))) \\
&=& \sum L_{\Ad_s(W_j(x))}(\Ad_s(W_j(x)) , 
\end{eqnarray*}
where we have used proposition 3.1 of \cite{R22}. But again the independence of the
choice of frame shows the invariance under the $\Ad$-action of $K$.
\end{proof}

\begin{corollary}
\label{cor7.9}
Let $G/K$ be the coadjoint orbit for $\mu^\dia \in \fg'$, and let $\nabla$ be any $G$-invariant connection on $\cT(G/K)$ that is compatible with $g_\dia$ and commutes with $J_\dia$. Then for any character $\chi$ of $K$ the Dirac operator $D^\nabla$ on $\cS(G/K, g_\dia,  \chi)$ is formally self-adjoint.
\end{corollary}

\begin{proof}
Because $K$ contains a maximal torus, the only element of $\fm$ that is invariant for the $\Ad$-action of $K$ is 0.
\end{proof}

We remark that when $G/K$ can be identified with a coadjoint orbit, there are usually many different coadjoint orbits to which it can be identified, and thus many different complex structures $J$ (and Riemannian metrics) that can be used when applying the above corollary.

From Lemma \ref{lem7.8} we see that the criterion of Corollary \ref{cor7.7} will be satisfied if and only if 
$(\sum_j L_{W_j} W_j)(e) = 0$. Let $\{Y_p\}$ be a $g_0$-orthonormal basis for $\fm$.
Then $\{S^{1/2}Y_p\}$ will be a $\Kil$-orthonormal basis for $\fm$, which we can extend to a $\Kil$-orthonormal basis
$\{X_j\}$ for $\fg$.
Then $\{\hat X_j\}$ is a standard module $\Kil$-biframe for $\cT(G/K)$, and so, as seen just before Theorem \ref{th2.10}, $\{(\hat X_j, S^{-1} \hat X_j)\}$ is a standard module $g_0$-frame for $\cT(G/K)$. By Proposition \ref{prop2.9} the criterion is equivalent to
$
0 = \sum_j L_{\hat X_j}(e)(S^{-1}\hat X_j(e))
$
Now as seen before Theorem \ref{th2.10}, 
$
{\hat X}(x) = -P\Ad_x^{-1}(X)
$
for any $X \in \fg$, so that $\hat X(e) = -PX$. Consequently $S^{-1}\hat X_p(e) = -S^{-1/2} Y_p$ for each $p$. In this way we obtain the following corollary, which is very similar to the criterion that Ikeda obtained for the spin case in the main theorem of \cite{Ikd}:

\begin{corollary}
\label{cor7.10}
Let $(G/K,g_0, J)$ be almost Hermitian, and let $\nabla$ be a $G$-invariant connection on $\cT(G/K)$ that is compatible with $g_0$ and commutes with $J$. Let $D^{\nabla}$ be the Dirac operator for $g_0$ and $J$ on $\cS(G/K, \chi)$ for a character $\chi$ of $K$. Let $L = \nabla - \nabla^c$. Then $D^{\nabla}$ is formally self-adjoint if and only if for one (and so for any) $g_0$-orthonormal basis $\{Y_p\}$ for $\fm$ we have
\[
\sum_p L_{\hat X_p}(e) (S^{-1/2}Y_p) = 0
\] 
where $X_p = S^{1/2}Y_p$ for each $p$.
\end{corollary}

For the essential self-adjointness of Dirac operators see, for example, section~9.4 of \cite{GVF} and section~4.1 of \cite{Frd}.

I thank John Lott and Mattai Varghese for independently bringing to my attention the paper \cite{Bsm}. In this paper connections which have non-zero torsion are considered, and in definition 1.9 certain modified Dirac operators are defined, and in theorem 1.10 these modified Dirac operators are shown to be self-adjoint. We can in the same way define self-adjoint modified Dirac operators. Within the setting of Theorem \ref{th7.1} let $\nabla = \nabla^c + L$ as before. Then from the definition of Dirac operators in terms of standard module frames given after Definition \ref{def6.7} we see that
\[
D^\nabla \psi = D^{\nabla^c} \psi + \sum_j \k(W_j) L_{W_j} \psi  .
\]
Let $M$ be the operator defined by $M \psi = \sum_j  \k(W_j) L_{W_j} \psi$. It is clearly a bounded operator on $L^2(G/K, \cS)$. Then on $\cS(G/K, \chi)$ we have $(D^\nabla)^* = (D^{\nabla^c})^* + M^*$. But we saw in Corollary \ref{cor7.5} that $D^{\nabla^c}$ is formally self-adjoint. From this we see that
$D^\nabla - (D^\nabla)^* = M - M^*$. Consequently, if we define the modified Dirac operator by $\tilde D^\nabla = D^\nabla - (1/2)(M-M^*)$, then it is easily seen that $\tilde D^\nabla$ is formally self-adjoint.

I also thank John Lott for bringing to my attention the paper \cite{Goe}. It assumes only that $K$ is a connected subgroup of $G$, and deals only with metrics on $G/K$ that are ``normal'', that is come from $G$-invariant metrics on $\fg$. The connections that are considered, which can have non-zero torsion, are quite similar to those used in \cite{Agr}. In the first two sections $G/K$ is assumed to be spin, and the Dirac operators are self-adjoint, for reasons that appear to be closely related to Corollary \ref{cor7.5}. In the next sections of \cite{Goe} $G/K$ is not assumed to be spin, but this is dealt with by tensoring the spinor representation of $\bC\ell(\fm)$ by suitable unitary representations of $K$. This appears to be related to the ``$\mathrm{spin}_K$'' structures of \cite{BIL}, but I have not explored this technique.

It would be interesting to know how all of the results of our paper relate to Connes' action principle for finding {\em the} Dirac operator from among all of the spectral triples that give a specified Riemannian metric \cite{Cn3}.  (See also theorem~11.2 and section~11.4 of \cite{GVF}.)  Of course, on the face of it Connes' theorem is for $\spin$ manifolds while many homogeneous spaces are not $\spin$.


\def\dbar{\leavevmode\hbox to 0pt{\hskip.2ex 
\accent"16\hss}d}
\providecommand{\bysame}{\leavevmode\hbox 
to3em{\hrulefill}\thinspace}
\providecommand{\MR}{\relax\ifhmode\unskip\space\fi MR } 
\providecommand{\MRhref}[2]{%
\href{http://www.ams.org/mathscinet-getitem?mr=#1}{#2} 
}
\providecommand{\href}[2]{#2}

}   

\end{document}